\documentclass{amsart}

\usepackage{latexsym,amscd,ifthen,color,multicol,hyperref}
\usepackage{verbatim}
\usepackage[normalem]{ulem} 
\usepackage{tikz}
\usepackage{amsmath}
\usepackage{amsfonts}
\usepackage{amssymb}
\usepackage{amsthm}
\usepackage{amsxtra}
\usepackage{dsfont}
\usepackage{enumitem}
\usepackage{mathrsfs}
\usepackage{mathtools}
\usepackage{tikz-cd}
\usepackage{multirow}
\usepackage{colortbl}
\usepackage{hhline}
\usepackage{color}

\numberwithin{equation}{section}

\theoremstyle{plain}
\newtheorem{theorem}{Theorem}[section]
\newtheorem{lemma}[theorem]{Lemma}

\newtheorem{proposition}[theorem]{Proposition}

\newtheorem{corollary}[theorem]{Corollary}

\theoremstyle{definition}
\newtheorem{definition}[theorem]{Definition}

\newtheorem{definitiontheorem}[theorem]{Definition-Theorem}
\newtheorem{definitionproposition}[theorem]{Definition-Proposition}

\newtheorem{notation}[theorem]{Notation}

\newtheorem{remark}[theorem]{Remark}
\newtheorem{question}[theorem]{Question}

\makeatletter              
\let\c@equation\c@theorem  
\makeatother

\newcommand{\lra}{\longrightarrow}

\newcommand{\Z}{\mathbb{Z}}

\newcommand{\N}{\mathbb{N}}

\newcommand{\kk}{\mathds{k}}
\newcommand{\mc}{\mathcal}

\newcommand{\mf}{\mathfrak}
\newcommand{\wh}{\widehat}
\newcommand{\defeq}{\vcentcolon=}

\newcommand{\tens}{\otimes}
\newcommand{\com}[2]{#1 _ { (#2)}}

\newcommand{\comm}[3]{#1 _ { (#2) (#3)}}

\newcommand{\ord}{\textup{ord}}
\newcommand{\End}{\textup{End}}
\newcommand{\Aut}{\textup{Aut}}
\newcommand{\Id}{\textup{id}}

\newcommand{\blank}{\underline{\hspace{2.5mm}}}

\newenvironment{enum}
	{\begin{enumerate}[ label=\textup{(\alph*)}] }
	{\end{enumerate}}

\makeatletter
\def\@tocline#1#2#3#4#5#6#7{\relax
  \ifnum #1>\c@tocdepth 
  \else
    \par \addpenalty\@secpenalty\addvspace{#2}%
    \begingroup \hyphenpenalty\@M
    \@ifempty{#4}{%
      \@tempdima\csname r@tocindent\number#1\endcsname\relax
    }{%
      \@tempdima#4\relax
    }%
    \parindent\z@ \leftskip#3\relax \advance\leftskip\@tempdima\relax
    \rightskip\@pnumwidth plus4em \parfillskip-\@pnumwidth
    #5\leavevmode\hskip-\@tempdima
      \ifcase #1
       \or\or \hskip 1em \or \hskip 2em \else \hskip 3em \fi%
      #6\nobreak\relax
    \dotfill\hbox to\@pnumwidth{\@tocpagenum{#7}}\par
    \nobreak
    \endgroup
  \fi}
\makeatother

\begin{document}

\title[On actions of doubles on finite dimensional algebras]{On actions of Drinfel'd doubles on finite dimensional algebras}

\author{Zachary Cline}
\address{Department of Mathematics, Temple University, Philadelphia, Pennsylvania 19122, USA}
\email{zcline@temple.edu}

\bibliographystyle{abbrv}

\begin{abstract}
      Let $q$ be an $n^{th}$ root of unity for $n > 2$ and let $T_n(q)$ be the Taft (Hopf) algebra of dimension $n^2$.
      In 2001, Susan Montgomery and Hans-J\"urgen Schneider classified all non-trivial $T_n(q)$-module algebra structures on an $n$-dimensional associative algebra $A$.  
      They further showed that each such module structure extends uniquely to make $A$ a module algebra over the Drinfel'd double of $T_n(q)$.
      We explore what it is about the Taft algebras that leads to this uniqueness, by examining actions of (the Drinfel'd double of) Hopf algebras $H$ ``close'' to the Taft algebras on finite-dimensional algebras analogous to $A$ above.
      Such Hopf algebras $H$ include the Sweedler (Hopf) algebra of dimension 4, bosonizations of quantum linear spaces, and the Frobenius-Lusztig kernel $u_q(\mf{sl}_2)$.
\end{abstract}



\maketitle

\section{Introduction} \label{sect:intro}

Throughout, let $ \kk $ be an algebraically closed field of characteristic 0. 
The unadorned tensor product, $\tens$, will mean $\tens_\kk$, and all algebraic structures will be over $\kk$ unless otherwise stated.

Our work is first motivated by the problem of finding solutions to the quantum Yang-Baxter equation, which provide a source of link invariants and play a role in the theory of quantum integrable systems \cite{jimboybe, kassel}.
For a vector space $V$, a map $c \in \Aut_\kk(V \tens V)$ is called a solution of the \emph{quantum Yang-Baxter equation} if the identity
\[
  (c \tens \Id_V)(\Id_V \tens c)(c \tens \Id_V) = (\Id_V \tens c)(c \tens \Id_V)(\Id_V \tens c)
\]
holds in $\Aut_\kk(V \tens V \tens V)$.
In \cite{quantumgroups}, Drinfel'd introduced the notion of \emph{quasitriangular} Hopf algebras, whose modules each lead to a solution of the quantum Yang-Baxter equation.
He also introduced the \emph{quantum double} of a finite-dimensional Hopf algebra $H$ (now called the \emph{Drinfel'd double} of $H$), denoted $D(H)$, which is a canonical quasitriangular Hopf algebra in which $H$ embeds. 
Thus, modules of a finite-dimensional Hopf algebra $H$ which admit an extension to the structure of a $D(H)$-module give solutions of the quantum Yang-Baxter equation.

Towards the second motivation of our work, take $A$ to be an associative algebra, and recall that an action of a group $G$ on $A$ by algebra automorphisms captures the symmetry of $A$ in a sense. 
However, we can capture more symmetry of $A$ if we consider more generally actions of a Hopf algebra $H$ on $A$ so that $A$ arises as an \emph{$H$-module algebra}.
For example, along with automorphisms, such an action can capture derivations of $A$; both group algebras and universal enveloping algebras of a Lie algebra are key examples of Hopf algebras.
Even more symmetry of $A$ can be investigated if that action can be extended so that $A$ admits a non-trivial $D(H)$-module algebra structure.

Thus, for the sake of both studying symmetries of associative algebras and for finding solutions of the quantum Yang-Baxter equation, we are interested in the question of when actions of a finite-dimensional Hopf algebra $H$ on $A$ leads to a non-trivial action of $D(H)$ on $A$.
In particular, we explore the question of when a group ($G$-)action on $A$ by algebra automorphisms can extend non-trivially to an action of a Hopf algebra $H$ on $A$, and when this action can then extend non-trivially to an action of $D(H)$ on $A$.

The scope of this paper is based on the work of Susan Montgomery and Hans-J\"urgen Schneider in \cite{montschneid} on actions of the $n$-dimensional \emph{Taft Hopf algebra}, $T_n(q)$, generated by a grouplike element $g$ and a $(g,1)$-skew primitive element $x$, subject to the relations
\[
  g^n = 1, \quad x^n = 0, \quad gx = qxg.
\]
Here, $n \in \N_{\geq 2}$ and $q$ is a primitive $n^{th}$ root of unity.
In \cite{montschneid}, Montgomery and Schneider classified the $n$-dimensional $T_n(q)$-module algebras with no nonzero nilpotent elements, for which $x$ does not act by zero.
In fact, $x$ acting by nonzero is exactly the condition that this module structure is \textit{inner-faithful}, i.e., that the action does not factor through any proper Hopf quotient of $T_n(q)$ (see, e.g., Corollary~\ref{cor:ifprims}).
Moreover, by Lemma~\ref{lem:dim_bound} below, the value $n$ is the smallest possible dimension of an inner-faithful $T_n(q)$-module algebra with no nonzero nilpotent elements. 
Their classification was the following.

\begin{theorem}\cite[Theorem~2.5]{montschneid} \label{thm:taft}
   Take $n \geq 2$.
   Let $A$ be an $n$-dimensional inner-faithful $T_n(q)$-module algebra with no nonzero nilpotent elements.
   Then there exists an element $u \in A$ and nonzero scalars $\beta, \gamma \in \kk$ such that $A = \kk[u]/(u^n - \beta)$, $g \cdot u = q u$, and $x \cdot u = \gamma 1_A$. \qed
\end{theorem}

By scaling $u$, we can assume without loss of generality that $u^n = 1_A$ in $A$ above.
Thus, $A$ is in fact~isomorphic as an algebra to the group algebra $\kk G$, where $G = G(T_n(q)) \cong \Z / n\Z$ is the group of grouplike elements of $T_n(q)$.
Moreover, note that since $G$ is abelian, $G \cong \widehat G$, the character group of $G$.
(This isomorphism is in general not unique.)
The action of $ \kk G \subseteq T_n(q)$ on $A \cong \kk G$ is induced by the character group:
Fix generators $g \in \widehat G$ and $u \in G$ so that $\langle g, u \rangle = q$; then, in $A \cong  \kk G$, we get that $g \cdot u^m = q^m u^m = \langle g, u^m \rangle u^m$.

Thus, Montgomery and Schneider classified all the inner-faithful actions of $T_n(q)$ on the group algebra of its grouplike elements $ \kk G(T_n(q))$, extending the action of $ \kk G(T_n(q))$ on itself as just described. 
We set the following notation.

\begin{notation}[$A(H)$] \label{not:ah}
   For a Hopf algebra $H$ with a finite abelian group of grouplike elements $G \defeq G(H)$, let $A(H)$ denote an inner-faithful $H$-module algebra that is isomorphic to $ \kk G$ as an algebra so that $ \kk G \subset H$ acts on $A(H) \cong \kk G$ as $\kk \widehat G$ does in the manner described above.
\end{notation}

Montgomery and Schneider showed further that for $n \geq 3$, each such action of $T_n(q)$ on $A(T_n(q))$ can be extended uniquely to an action of $D(T_n(q))$ on $A(T_n(q))$; we recall the details of their result in Theorem~\ref{thm:taftext}.
Therefore, each module algebra $A(T_n(q))$ gives a solution to the quantum Yang-Baxter equation, and the symmetries of $A(T_n(q))$ coming from the action of $D(T_n(q))$ are, in a sense, determined uniquely by the symmetries coming from the action of $T_n(q)$.
Motivated by their work, we investigate the following questions.
\begin{question} \label{questions}
   Let $H$ be a finite-dimensional Hopf algebra with an abelian group of grouplike elements.
   \begin{enumerate}
      \item
         Do the module algebra structures $A(H)$ as described in Notation~\ref{not:ah} exist?
   \end{enumerate}
   If (a) is affirmative, then:
   \begin{enumerate}[resume]
      \item
         What are the possible $H$-module structures on $A(H)$?
      \item
         What are the possible $D(H)$-module algebra structures on $A(H)$ extending that in (b)?
         How many extensions are there? 
         In particular, is there a unique extension as in the case of the Taft algebras?
   \end{enumerate}
\end{question}

\begin{remark}
  The first case to consider is, naturally, the case $H = \kk G$ for $G$ a finite abelian group.
  Here, $A(H) = H$ with the action coming from any faithful action of $\widehat G \cong G$ on $\kk G$ by algebra automorphisms, which addresses \mbox{Question~\ref{questions}(a,b).}
  Note that $D(\kk G) \cong \kk G \tens \kk G$ as Hopf algebras with the tensor product Hopf algebra structure.
  The second copy of $\kk G$ corresponds to the original $H$, and the first copy corresponds to the dual $(\kk G)^* \cong \kk G$. 
  Thus, any extension of an action of $\kk G$ on $A(\kk G)$ to one of $D(\kk G)$ on $A(\kk G)$ is given by any other action (not necessarily faithful) of $\widehat G \cong G$ on $\kk G$ by algebra automorphisms.
\end{remark}

Because the answers to Question~\ref{questions} are interesting for the Taft algebras $T_n(q)$, we will answer these questions for some pointed, finite-dimensional Hopf algebras related to Taft algebras.
In Section~\ref{sect:basics}, we provide background information pertaining to actions of pointed Hopf algebras and their Drinfel'd doubles that will be used throughout.
Section~\ref{sect:taft} goes over the case of the Taft algebras in more detail, and gives an answer to Question~\ref{questions}(c) for the Sweedler algebra $T_2(-1)$.
Section~\ref{sect:qls} is dedicated to a family of coradically graded Hopf algebras, $H_n(\zeta,m,t)$, for which the Taft algebras are a subclass; these Hopf algebras arise as bosonizations of quantum linear spaces from Andruskiewitsch and Schneider's work \cite{as-qls}.
Explicit computations are given for the dual $H_n(\zeta,m,t)^*$, with the dual pairing given, and for $D(H_n(\zeta,m,t))$ before addressing Question~\ref{questions}.
A non-trivial lifting of $H_4(\zeta,1,2)$, namely the generalized Taft algebra $T(4,2,1)$, is the subject of Section~\ref{sect:gentaft}.
Again, explicit computations of the dual and double are given for $T(4,2,1)$.
It is known that a Taft algebra can be considered as the positive Borel part of the Frobenius-Lusztig kernel $u_q(\mf{sl}_2)$.
Section~\ref{sect:uqsl2} answers Question~\ref{questions} for the full small quantum group $u_q(\mf{sl}_2)$, while Appendix~\ref{sect:uqdouble} is devoted to the computation of the presentations of $u_q(\mf{sl}_2)^*$ and of $D(u_q(\mf{sl}_2))$ needed specifically for this work, which may be of independent interest. 
In particular, our method for answering Question~\ref{questions} for $H = u_q(\mf{sl}_2)$ works best when the comultiplication of its Drinfel'd double is uncomplicated.
Our main results are summarized as follows.

\begin{theorem} \label{thm:main}
  Consider the finite-dimensional pointed Hopf algebras $T_n(q)$, $T_2(-1)$, $H_n(\zeta,m,t)$, $T(4,2,1)$, and $u_q(\mf{sl}_2)$ discussed above.
  Then, Question~\ref{questions} is addressed for these Hopf algebras, as detailed in Table~\ref{tab:results}.
\end{theorem}
The results of Montgomery and Schneider for $T_n(q)$ are included in Table~\ref{tab:results} for comparison, and the proof of Theorem~\ref{thm:main} is the focus of the rest of the work.

\begin{center}
  \begin{table}[h!]
    \begin{tabular}{ |c|c|c|c| }
      \hline
      \cellcolor[gray]{.9} \textbf{$H$}
        & \cellcolor[gray]{.9} \textbf{Actions of $H$ on $A(H)$}
        & \cellcolor[gray]{.9}  \textbf{Extension to actions of $D(H)$}
        & \cellcolor[gray]{.9} \textbf{\# / Parametrization} \\
        \cellcolor[gray]{.9} (gens. of $H$)
        & \cellcolor[gray]{.9} ( $\forall H$, $A(H)$ is gen. by $u$)
        & \cellcolor[gray]{.9}
        & \cellcolor[gray]{.9} \textbf{of extns. to} \\
        \cellcolor[gray]{.9} (gens. of $H^*$)
        & \cellcolor[gray]{.9} on Question~\ref{questions}(a,b)
        & \cellcolor[gray]{.9} on Question~\ref{questions}(c)
        & \cellcolor[gray]{.9} \textbf{actions of $D(H)$} \\      
      \hline
      $T_n(q)$
        & $g \cdot u = qu, \ \ x \cdot u = \gamma 1 $
        & $G \cdot u = q^{-1}u, \ \ X \cdot u = \lambda u^2$
        & \multirow{3}{*}{$1$} \\
        $(g,x)$
        & $0 \neq \gamma \in \kk$
        & $\lambda \in \kk, \ \gamma \lambda = q^{-1}-1$
        & \\
        $(G,X)$
        & [Theorem \ref{thm:taft}]
        & [Theorem \ref{thm:taftext}]
        & \\
      \hline
      \cellcolor[gray]{.9} $T_2(-1)$
        & \cellcolor[gray]{.9} $g \cdot u = -u, \ \ x \cdot u = \gamma 1$
        & \cellcolor[gray]{.9} $G \cdot u = -u, \ \ X \cdot u = \lambda 1 $
        & \cellcolor[gray]{.9} \\
        \cellcolor[gray]{.9} $(g,x)$
        & \cellcolor[gray]{.9} $0 \neq \gamma \in \kk$
        & \cellcolor[gray]{.9} $\lambda \in \kk$
        & \cellcolor[gray]{.9}  \\
        \cellcolor[gray]{.9} $(G,X)$
        & \cellcolor[gray]{.9}[Theorem \ref{thm:taft}]
        & \cellcolor[gray]{.9} [Proposition \ref{prop:sweedlerext}]
        & \cellcolor[gray]{.9} \multirow{-3}{*}{$\kk$} \\
      \hline
      \multirow{4}{*}{$H_n(\zeta,m,t)$}
          & exists if $\text{gcd}(mt, n) = m$: 
          & $Y \cdot u = \zeta^d u,$
          & \\
          & $y \cdot u = \zeta u,$
          & $ X \cdot u = \delta u^{n+1-t}$,
          & \multirow{2}{*}{$t$, if $2m \neq n$ } \\
          & $x \cdot u = \gamma u^{t+1}$,
          &  $m \equiv -dt \pmod n$,
          & \\
          & $0 \neq \gamma \in \kk$
          & $\delta \in \kk$,
          & \multirow{2}{*}{$t \times \kk$, if $2m = n$ } \\
        $(y,x)$
          & \multirow{2}{*}{[Proposition \ref{prop:minexistence}]}
          & $\gamma \delta = \frac {\zeta^{-m} - 1}{(n-t)_{\zeta^m}} $ if $2m \neq n$
          & \\
        $(Y,X)$
          &
          & [Theorem \ref{thm:ahnzmt}]
          & \\
      \hline
      \cellcolor[gray]{.9} $T(4,2,1)$
        & \cellcolor[gray]{.9} $g \cdot u = qu, \ \ x \cdot u = \gamma u^3$
        & \cellcolor[gray]{.9}
        & \cellcolor[gray]{.9} \\
        \cellcolor[gray]{.9} $(g,x)$
        & \cellcolor[gray]{.9} $\gamma \in \kk, \ \gamma^2 = 2q$ 
        & \cellcolor[gray]{.9} \multirow{-2}{*}{(None)}
        & \cellcolor[gray]{.9} \\
        \cellcolor[gray]{.9} $(G,X)$
        & \cellcolor[gray]{.9} [Proposition \ref{prop:at421}]
        & \cellcolor[gray]{.9} [Proposition \ref{prop:t421extensions}]
        & \cellcolor[gray]{.9} \multirow{-3}{*}{$0$}  \\
      \hline
        \multirow{4}{*}{$u_q(\mf{sl}_2)$}
          & $K \cdot u = q^2 u,$
          & $a \cdot u = qu, \ \ b \cdot u = (q - q^{-1}) 1,$
          & {} \\
          & $F \cdot u = \gamma 1,$
          & $c \cdot u = 0, \ \ d \cdot u = q^{-1} u$,
          & {} \\
          & $E \cdot u = \delta u^2$
          & --- or ---
          & \\
          & $\gamma, \delta \in \kk, \ \gamma \delta = -q$ 
          & $a \cdot u = q^{-1}u, \ \ b \cdot u = 0, $
          & {} \\
        $(K,E,F)$
          & \multirow{2}{*}{[Proposition \ref{prop:montuq}]}
          & $c \cdot u = (q - q^{-1}) 1, \ \ d \cdot u = q u $
          & {} \\
        $(a,b,c,d)$
          & 
          & [Theorem \ref{thm:uqextension}]
          & \multirow{-6}{*}{$2$} \\            
      \hline
    \end{tabular}
    \medskip
    \caption{Summary of Main Results} \label{tab:results}
  \end{table}
\end{center}

\vspace{-.2in}

Now we recall related results in the literature that may be of interest.
In \cite{cfm-ydcats}, Cohen, Fischman, and Montgomery examine conditions on a Hopf algebra $H$ and left $H$-module $H$-comodule algebra $A$ under which $A$ can be realized as a $D(H)$-module algebra.
In particular, they show that if $H$ has a bijective antipode and either (i) $A$ is a faithful $A \# H$-module, or (ii) $A / A^{coH}$ is $H$-Galois and $A$ is $H$-commutative, then $A$ is a $D(H)$-module algebra.
Chen and Zhang classified all $D(T_2(-1))$-module algebras of dimension 4 up to isomorphism as $D(T_2(-1))$-modules in \cite{chen-zhang}, in particular giving all $D(T_2(-1))$-module algebra structures on $M_2(\kk)$.
In \cite{kw-quivers}, Kinser and Walton examine actions of Taft algebras on path algebras of quivers, and extend such actions to $D(T_n(q))$.

\medskip
We also wish to briefly mention some questions and further directions.
First, Question~\ref{q:nichols} points out that there are many ways to generalize Taft algebras, such as examining quantum linear spaces of higher rank over abelian, non-cyclic groups, or by considering more generally Nichols algebras (of Cartan type) in the Yetter-Drinfeld category ${}_G^G \mc{YD}$ for some abelian group $G$.
It is possible that similar results hold in one or more of these cases.
Question~\ref{q:borel} pertains to the relationship between the number of ways to extend an action of $u_q(\mf g)$ on $A(u_q(\mf g))$ to an action of $D(u_q(\mf g))$ and the number of ways to extend an action of a Borel subalgebra of $u_q(\mf g)$ to an action of its double.
It arises from the very limited data we have, i.e. for $\mf g = \mf{sl}_2$.


\section{Definitions and basic concepts} \label{sect:basics}

For an exposition of coalgebras, Hopf algebras, and their representations, see \cite{montgomery} and \cite{radford}.
Here, we recall some of the basic notions and specify the notation used throughout.
We will denote the comultiplication and counit maps of a coalgebra by $\Delta$ and $ \epsilon $, respectively.
The set of \emph{grouplike elements} of a coalgebra $G(C)$ are the nonzero elements $c$ such that $\Delta(c) = c \tens c$.
For $g,h \in G(C)$, the space of \emph{$(g,h)$-skew primitive elements} $P_{g,h}(C)$ is the set of elements $c \in C$ such that $\Delta(c) = g \tens c + c \tens h$.
The symbol $S$ will be used for the antipode of a Hopf algebra $H$.
Sweedler notation will also be used for the comultiplication throughout: we will write $\com a 1 \tens \com a 2$ for $ \Delta(a)$.
For an algebra $A$ and a left $A$-module $M$, we will denote the action of $a \in A$ on $m \in M$ by $a \cdot m$.
For a coalgebra $C$ and a left $C$-comodule $M$, we will signify the coaction by $\rho: M \to C \tens M$, and use the modified Sweedler notation $\rho(m) = \com m {-1} \tens \com m 0$ for $m \in M$.

Recall that a coalgebra $C$ is called \emph{simple} if its only subcoalgebras are $0$ and $C$.
The \emph{coradical} of $C$ is the (direct) sum of its simple subcoalgebras and is denoted $C_0$.
A coalgebra (or bialgebra, Hopf algebra) is called \emph{pointed} if all its simple subcoalgebras are $1$-dimensional, or equivalently, if all its simple left (or right) comodules are $1$-dimensional.
In fact, $C$ is \emph{pointed} if and only if $C_0 =  \kk G(C)$.
It is well-known that any bialgebra generated by grouplike and skew-primitive elements is pointed.
Andruskiewitsch and Schneider conjecture that, conversely, all finite-dimensional pointed Hopf algebras, $H$, over an algebraically closed field of characteristic $0$ are generated by grouplike and skew-primitive elements, \cite[Conjecture~5.7]{as-pha}; Angiono verified this conjecture in the case when $G(H)$ is abelian, \cite[Theorem~2]{angiono}.
The \emph{coradical filtration} of $C$ is defined inductively by $C_0$ being the coradical and $C_i = \Delta^{-1}( C_{i-1} \tens C + C \tens C_0)$ for all $i > 0$.
We denote the associated graded coalgebra by $\text{gr}(C)$.
A graded coalgebra (or bialgebra, Hopf algebra) $C = \bigoplus_{i \geq 0} C(i)$ is called \emph{coradically graded} if $\bigoplus_{j = 0}^i C(j) = C_i$ for all $i \geq 0$, where $\{C_i\}$ is the coradical filtration.
For a Hopf algebra $H$, if $H_0$ is a Hopf subalgebra, then $\text{gr}(H)$ is coradically graded, and $H$ is called a \emph{lifting} of $\text{gr}(H)$.

In Section~\ref{subsect:innerfaithful}, we recall some facts about inner-faithful module algebras over pointed Hopf algebras.
Section~\ref{subsect:qsymbols} introduces the $q$-binomial coefficients and related symbols.
We give more details about the~Drinfel'd double construction in Section~\ref{subsect:drinfelddouble}.
In Section~\ref{subsect:duality}, we record the definition of a perfect duality between two Hopf algebras.
In Section~\ref{subsect:ydmodules}, the definition of Yetter-Drinfel'd modules and bosonizations are recalled.

\subsection{Inner-faithful module algebras} \label{subsect:innerfaithful}

For $H$ a Hopf algebra, an \emph{$H$-module algebra} $A$ is an algebra (or monoid) in the monoidal category of $H$-modules.
In other words, $A$ is an $H$-module such that $h \cdot (ab) = (\com h 1 \cdot a) (\com h 2 \cdot b)$ and $h \cdot 1_A = \epsilon(h) 1_A$ for all $h \in H$ and $a,b \in A$.
Throughout, we consider module algebras over some pointed Hopf algebras that are faithful in the following sense.

\begin{definition}
   Let $H$ be a Hopf algebra and $M$ a left $H$-module.
   We say that $M$ is an \emph{inner-faithful} $H$-module, or that the action of $H$ on $M$ is \emph{inner-faithful} provided $I \cdot M \neq 0$ for any nonzero Hopf ideal $I$ of $H$.
   In other words, the action of $H$ on $M$ is inner-faithful provided the action on $M$ does not factor through any proper Hopf quotient of $H$. 
   If $A$ is an $H$-module algebra such that the action of $H$ on $A$ is inner-faithful, we call $A$ an \emph{inner-faithful} $H$-module algebra.
\end{definition}

Clearly, if the action of $H$ on $M$ is faithful, then it is inner-faithful.
Since all of the Hopf algebras we will consider in this work are pointed, the following results will be useful.

\begin{lemma} \label{lem:primsinideals}
   Let $H$ be a pointed Hopf algebra and $I$ a nonzero Hopf ideal of $H$. 
   Then $I$ contains a nonzero element of $P_{g,1}(H)$ for some $g \in G(H)$. 
\end{lemma}

   \begin{proof}
      Consider the projection map $f: H \to H/I$.
      Since $I \neq 0$, $f$ is not injective.
      Therefore, by \cite[6.1.1]{takeuchi}, we can fix some $g,h \in G(H)$, with $f \mid _{P_{g,h}(H)}$ not injective.
      Choose nonzero $x \in P_{g,h}(H)$ such that $f(x) = 0$ (i.e. $x \in I$), and take $x' = xh^{-1}$.
      Then $x' \in P_{gh^{-1},1}(H) \cap I$ and $x' \neq 0$, or else $x = x' h = 0$.
   \end{proof}

\begin{corollary} \label{cor:ifprims}
   Let $H$ be a pointed Hopf algebra and $A$ an $H$-module algebra.
   Then the action of $H$ on $A$ is inner-faithful if and only if for each $g \in G(H)$ and nonzero $x \in P_{g,1}(H)$ we have that $x \cdot A \neq 0$.
\end{corollary}

   \begin{proof}
      First, suppose the action is not inner-faithful.
      Choose a nonzero Hopf ideal $I$ such that $I \cdot A = 0$.
      Then by Lemma~\ref{lem:primsinideals}, $I$ contains some $(g,1)$-skew primitive element $x$ for some $g \in G(H)$.
      Thus, $x \cdot A = 0$.
      
      Now suppose there exists $g \in G(H)$ and nonzero $x \in P_{g,1}(H)$ such that $x \cdot A = 0$.
      Since the principal ideal $(x)$ is a Hopf ideal, we have obtained a nonzero Hopf ideal which acts by zero on $A$, and thus, the action of $H$ on $A$ is not inner-faithful.
   \end{proof}

These results actually give us a lower bound on the $\kk$-vector space dimension of inner-faithful module algebras with no nonzero nilpotent elements.

\begin{lemma} \label{lem:dim_bound}
   Suppose that a finite group $G$ acts faithfully by algebra automorphisms on a finite-dimensional $\kk$-algebra $A$ with no nonzero nilpotent elements. 
   Then 
   \[
      \dim_\kk(A) \geq \max \{\ord(g): g \in G \} .
   \]
\end{lemma}

   \begin{proof}
      Let $ g \in G $ and $n = \ord(g)$. 
      Since $\langle g \rangle$ is finite abelian, the action of $g$ on $A$ is diagonalizable with
      \[
          A = \bigoplus_ { i = 0 } ^ { n - 1 } A_i,  \hspace{15mm} %
          A_i = \{ a \in A : g \cdot a = q^i a \},
      \]
      where $q$ is a fixed primitive $n^{th}$ root of unity. 
      Because $\ord(g) = n$, and the action is faithful, there exists $j$ such that $( j , n ) = 1$ and $A_j \neq 0$. 
      Without loss of generality, by choosing a different $q$, we can take $j = 1$. 
      Choose nonzero $u \in A_1$. 
      Since $A$ has no nonzero nilpotent elements, $u^i \neq 0$ for all $i$. 
      Also, $g \cdot u^i = q^i u^i $ for all $i$, showing that $u^i \in A_i$. 
      Thus, $A_i \neq 0$ for all $i$. 
      Therefore, $\dim_\kk(A) \geq n$.
   \end{proof} 
   
\begin{remark} \label{rem:astructure}
  For any Hopf algebra $H$, Lemma~\ref{lem:dim_bound} shows that if $G(H)$ is cyclic of order $n$, then the smallest possible dimension of an inner-faithful $H$-module algebra with no nonzero nilpotent elements is $n$, and that if such a lower bound is met, then these $H$-module algebras would be exactly $A(H)$ as in Notation~\ref{not:ah}.
   Fix a generator $g \in G(H)$.
   Then, for a generator $u \in A(H)$ such that $A(H) \cong \kk[u]/(u^n - 1)$, there is a primitive $n^{th}$ root of unity $q \in \kk$ with $g \cdot u = q u$.
   Alternatively, for a fixed $q$, we can choose $u \in A(H)$ such that $g \cdot u = qu$ and $A(H) = \kk[u]/(u^n - 1)$.
   Here, we write the eigenspaces of the $g$-action
   \[
      A_i = \{ a \in A: g \cdot a = q^i a \},
   \]
   noting that $A = \bigoplus _{i=0}^{n-1} A_i$ and $A_i =\kk u^i$.
   We will use this notation throughout.
\end{remark}

\subsection{$q$-Symbols} \label{subsect:qsymbols}

In all of the Hopf algebras $H$ that we consider in this work, there will be a relation of the form $yx = qxy$, for $q \in \kk$ where $x,y \in H$.
It is thus helpful to consider the \emph{quantum binomial coefficients}, ${\binom n m}_q$, which are defined using any $x,y$ such that $yx = qxy$ by
\begin{equation} \label{eq:skewbinom}
   (x + y)^n = \sum_{m=0}^n {\binom n m}_q x^{n-m} y^m.
\end{equation}

The $q$-binomial coefficients are related to the following symbols.
For any integer $n \geq 0$, set
\begin{gather*}
  (n)_q \defeq 1 + q + q^2 + \ldots + q^{n-1} = \frac{q^n - 1}{q-1} \quad \text{(if $q \neq 1$)}; \\
  (n)_q! \defeq (1)_q (2)_q \cdots (n)_q = \frac{(q-1)(q^2-1) \cdots (q^n-1)}{(q-1)^n} \quad \text{ (if $q \neq 1$)}.
\end{gather*}
By convention, we also define $(0)_q! = 1$.

The relationship between these symbols and $q$-binomial coefficients is given by \cite[Proposition~7.2.1(a)]{radford}:
If $(n-1)_q! \neq 0$, then one obtains that
$
  {\binom n m}_q = \frac {(n)_q!} {(m)_q! (n-m)_q!}.
$

We also have the following variation, which will be useful for the computation of $D(u_q(\mf{sl}_2))$ in Appendix~\ref{sect:uqdouble}.
Let $q \neq \pm 1 \in \kk$.
For any integer $n$, set
\begin{gather*}
  [n]_q = \frac {q^n - q^{-n}}{q - q^{-1}}
    = q^{n-1} + q^{n-3} + \cdots + q^{-n+1}.
\end{gather*}
For integers $0 \leq m \leq n$, set $[m]_q! = [1]_q [2]_q \cdots [m]_q$.
The relationship between these and the symbols $(k)_q$ defined above is given by
   $[n]_q = q^{-(n-1)} (n)_{q^2}$ and
   $[n]_q! = q^{-n(n-1)/2} (n)_{q^2}!$.

\subsection{The Drinfel'd double} \label{subsect:drinfelddouble}

We remind the reader of the construction of the Drinfel'd double of a finite-dimensional Hopf algebra.
Recall the transpose actions of a Hopf algebra $H$ on its dual $H^\circ$:
\[
   \langle a \succ p, b \rangle \defeq \langle p, ba \rangle, \quad 
   \langle p \prec a, b \rangle \defeq \langle p, ab \rangle, \quad 
   \text{for }a, b \in H,\  p \in H^\circ .
\]
When $H$ is finite-dimensional, $H^\circ = H^*$ is a Hopf algebra with multiplication given by $\Delta^*$ and comultiplication given by $m^*$. 
Thus, for $p \in H^*$, $\langle p, ab \rangle = \langle p, m(a \otimes b) \rangle = \langle m^*(p), a \otimes b \rangle = \langle p_{(1)}, a \rangle \langle p_{(2)}, b \rangle$.
Therefore, $a \succ p = \langle p_{(2)}, a \rangle p_{(1)}$ and $p \prec a = \langle p_{(1)}, a \rangle p_{(2)}$. 
Combining these two facts gives
\begin{equation*}
  a \succ p \prec b = \langle \com{p}{1}, b \rangle \langle \com{p}{3}, a \rangle \com{p}{2}.
\end{equation*}

\begin{definition}
   Let $H$ be a finite-dimensional Hopf algebra with antipode $S$. 
   (Recall that the antipode $S$ is then necessarily invertible.) 
   The \emph{Drinfel'd double}, $D(H)$, of $H$, is the Hopf algebra with coalgebra structure given by the tensor product coalgebra structure 
   \begin{equation} \label{eq:doublecomult}
      D(H) = H^{* cop} \otimes H,
   \end{equation}
   with multiplication given by
   \begin{equation} \label{eq:doublemult} 
      (p \otimes a)(q \otimes b) 
        ~=~ p \left( a_{(1)} \succ q \prec S^{-1}(a_{(3)})\right) \otimes a_{(2)}b 
        ~=~ \langle \com q 1, S^{-1}(\com a 3) \rangle \langle \com q 3, \com a 1 \rangle p  \; \com q 2 \tens \com a 2 b,
   \end{equation}
   with unit $\epsilon \tens 1$, and with antipode
   \begin{equation*} 
      S_{D(H)}(p \otimes a) = (\epsilon \otimes S(a)) (p \circ S^{-1} \otimes 1) \quad \text{for } p \in H^*, a \in H.
   \end{equation*}
   Simple tensors in $D(H)$ are written as $p \bowtie a$.
\end{definition}

Note that both $H$ and $H^{* cop}$ embed in $D(H)$, and we will think of elements of the former two as elements of the latter, by identifying $p \bowtie 1$ with $p$ and $\epsilon \bowtie a$ with $a$.
These identifications are justified by the following.

\begin{lemma} \label{lem:gens}
   Let $ H $ be a finite-dimensional Hopf algebra. 
   Then for $p, q \in H^*$ and $a, b \in H$, we have the following identities in $D(H)$: 
     $(p \bowtie 1)(\epsilon \bowtie a) = p \bowtie a$,
     $(p \bowtie 1)(q \bowtie 1) = pq \bowtie 1$,
     $(\epsilon \bowtie a)(\epsilon \bowtie b) = \epsilon \bowtie ab $,
     $S_{D(H)}(p \bowtie 1) = S_{H^{* cop}}(p) \bowtie 1$, and 
     $S_{D(H)}(\epsilon \bowtie a) = \epsilon \bowtie S_H(a)$.
   
   As a consequence, if $\{ a_i \}_{i = 1} ^n$ is a set of generators for $H$ and $\{ p_i \}_{i = 1} ^m$ is a set of generators for $H^*$, then $\{ p_i \bowtie 1 \}_{i = 1} ^m \cup \{\epsilon \bowtie a_i \}_{i = 1} ^n$ generates $D(H)$ as an algebra. 
   \qed
\end{lemma} 

From now on, we suppress the $\bowtie$ notation. 
It is clear that the relations between generators of $H$ and $H^*$ will also be relations in $D(H)$. 
Thus, to achieve an algebra presentation of $D(H)$, it remains to show how elements of $H$ move past those of $H^*$.
We will compute relations giving this ``commutation'' between elements of $H$ and $H^*$ using the following consequence of \eqref{eq:doublemult}: 
For any $p \in H^*$ and $a \in H$, we have in $D(H)$:
\begin{equation} \label{eq:doublemixed}
   ap ~=~ (\com a 1 \succ p \prec S^{-1}(\com a 3))  \com a 2 ~=~ \langle \com p 1, S^{-1}(\com a 3) \rangle \langle \com p 3, \com a 1 \rangle \com p 2  \com a 2.
\end{equation}
The explicit computation of the double of many finite-dimensional, pointed Hopf algebras will be given later in this article (see Section~\ref{subsubsect:hnzmtdouble}, Section~\ref{sect:t421double}, and Appendix~\ref{sect:uqdouble}).

\subsection{Perfect dualities} \label{subsect:duality}

For computing presentations of Drinfel'd doubles, we will first need presentations of dual Hopf algebras, in such a way that we know the dual pairing.
One helpful way for thinking about dual Hopf algebras is perfect dualities, which we recall from \cite[Definition~V.7.1]{kassel}.
Let $H$ and $K$ be Hopf algebras and $\langle \ , \  \rangle$ a bilinear form on $H \times K$.
We say $H$ and $K$ are \emph{in duality}, or that the bilinear form induces a duality between them, if the following hold for any $u,v \in H$ and $x,y \in K$:
\begin{equation}
  \begin{gathered}
    \langle uv, x \rangle = \langle u, \com x 1 \rangle \langle v, \com x 2 \rangle, \quad%
      \langle u, xy \rangle = \langle \com u 1, x \rangle \langle \com u 2, y \rangle, \\ 
    \langle 1, x \rangle = \epsilon_K(x), \quad %
      \langle u, 1 \rangle = \epsilon_H(u), \quad%
      \langle S_H(u), x \rangle = \langle u, S_K(x) \rangle. 
  \end{gathered} \label{eq:duality}
\end{equation}
With $\phi: H \to K^*$ and $\psi: K \to H^*$ defined by $\phi(u)(x) = \langle u, x \rangle = \psi(x)(u)$, we say the duality between $H$ and $K$ is \emph{perfect} if $\phi$ and $\psi$ are injective.
Observe that a perfect duality between finite-dimensional Hopf algebras induces an isomorphism $K \cong H^*$.

\subsection{Yetter-Drinfel'd modules and bosonizations} \label{subsect:ydmodules}

Let $H$ be a Hopf algebra. 
A \emph{(left-left) Yetter-Drinfel'd module} $M$ over $H$ is simultaneously a left $H$-module and a left $H$-comodule, satisfying the compatibility condition
\[
  \rho(h \cdot m) = \com h 1 \com m {-1} S(\com h 3) \tens \com h 2 \cdot \com m 0,
\]
for all $h \in H$ and $m \in M$.
We will denote the category of Yetter-Drinfel'd modules over $H$ by ${}_H^H \mc{YD}$.
If $H =\kk \Gamma$ is the group algebra of a group $\Gamma$, we will write ${}_\Gamma^\Gamma \mc{YD}$ for ${}_{\kk \Gamma}^{\kk \Gamma} \mc{YD}$.

Without going into further detail at this time, we remark that ${}_H^H \mc{YD}$ is a braided monoidal category, with braiding $c_{M,N}: M \tens N \to N \tens M$ given by
\begin{equation} \label{eq:ydbraiding}
   c_{M,N}(m \tens n) = \com m {-1} \cdot n \tens \com m 0.
\end{equation}

Since ${}_H^H \mc{YD}$ is a \textit{braided} monoidal category, the braiding \eqref{eq:ydbraiding} allows us to define Hopf algebras in ${}_H^H \mc{YD}$, typically called \emph{braided Hopf algebras}.
First, a \emph{braided bialgebra $B$ in ${}_H^H \mc{YD}$} is simultaneously an \emph{algebra} and \emph{coalgebra} in ${}_H^H \mc{YD}$ such that the morphisms $\Delta$ and $\epsilon$ are also algebra maps.
Here, the algebra structure of $B \tens B$ is defined using the braiding $c_{B,B}$ in place of the typical twist map $\tau$.
If the identity of a braided bialgebra $B$ has a convolution inverse $S$, which is also a morphism in ${}_H^H \mc{YD}$, then $B$ is called a \emph{braided Hopf algebra in ${}_H^H \mc{YD}$}.

If $B$ is a braided Hopf algebra in ${}_H^H \mc{YD}$, then in particular, $B$ is a left $H$-module algebra and a left $H$-comodule coalgebra.
Thus, $B \tens H$ is a $\kk$-algebra and a $\kk$-coalgebra via the smash product and smash coproduct structures, respectively.
By combining these structures, and using the antipode of $B$ and $H$, we get that $B \tens H$ is in fact a Hopf algebra over $\kk$, which we describe as follows.

\begin{definitiontheorem}[{\cite[Theorems~11.6.7, 11.6.9]{radford}}]
   Let $H$ be a Hopf algebra over $\kk$ and let $B$ be a braided Hopf algebra in ${}_H^H \mc{YD}$.
   Then $B \tens H$ is a Hopf algebra over $\kk$ with 
   \begin{itemize}
      \item
         unit $1_B \tens 1_H$,
      \item
         multiplication $(a \tens h)(b \tens k) = a (\com h 1 \cdot b) \tens \com h 2 k$,
      \item
         counit $\epsilon(b \tens h) = \epsilon_B(b) \epsilon_H(h)$,
      \item
         comultiplication $\Delta(b \tens h) = (\com b 1 \tens \comm b 2 {-1} \com h 1) \tens (\comm b 2 0 \tens \com h 2)$,
      \item
         and antipode $S(b \tens h) = (1 \tens S_H(\com b {-1} h))(S_B(\com b 0) \tens 1)$.
   \end{itemize}
   This Hopf algebra is called the \emph{bosonization} or \emph{biproduct} of $B$ and $H$, and is denoted by $B \# H$. \qed
\end{definitiontheorem}

Bosonizations have become an essential tool in the classification of pointed Hopf algebras, thanks to Radford's abstract characterization of those Hopf algebras that can be realized as bosonizations \cite[Theorem~3]{r-biproducts}.
Andruskiewitsch and Schneider have used Radford's result as a launching point for a very active program of classifying finite-dimensional pointed Hopf algebras \cite{a-ofdha, as-pha}.

For any $V \in {}_H^H \mc{YD}$, there is a canonical graded braided Hopf algebra $\mf B(V) \in {}_H^H \mc{YD}$, called a \emph{Nichols algebra}. 
These were first discovered by Warren D. Nichols and appeared in \cite{nichols}.
For a current survey of the Nichols algebras pertinent to the classification program of finite-dimensional pointed Hopf algebras, see the work of Andruskiewitsch and Angiono \cite{nicholsalgbible}.
If $V$ is a braided vector space of type $(A_1)^{\times \theta}$, then the Nichols algebra $\mf B(V)$ is called a \emph{quantum linear space over $H$} when $\mf B(V) \in {}_H^H \mc{YD}$ \cite{as-qls}.
We study these more in Section~3.

\section{The Taft algebras} \label{sect:taft}

Let $n \geq 2$ and $q \in \kk$ be a primitive $n^{th}$ root of unity.
Recall that the \emph{Taft algebra} $T_n(q)$ is generated by a grouplike element $g$ and a $(g,1)$-skew primitive element $x$, satisfying the following relations:
\[
   g^n = 1, \ \  x^n = 0, \ \ gx = qxg.
\]
Note that $\dim_\kk(T_n(q)) = n^2$.

In this section, we will consider Question~\ref{questions} in the Introduction for the Taft algebras $T_n(q)$.
Recall that Montgomery and Schneider have already answered Question~\ref{questions}(a,b) for actions of the Taft algebras $T_n(q)$ on the algebra $A(T_n(q))$ given in Notation~\ref{not:ah}; see Theorem~\ref{thm:taft}. 
They further answered Question~\ref{questions}(c) on actions of the double $D(T_n(q))$ on $A(T_n(q))$ for the case $n > 2$ as recalled below.

\begin{lemma} \cite[Lemma~4.4]{montschneid} \label{lem:taftdouble}
  The Hopf algebra $D(T_n(q))$ is generated by grouplike elements $g$ and $G$, a $(g,1)$-skew primitive element $x$, and a $(1,G)$-skew primitive element $X$, subject to the relations
  \begin{gather*}
     g^n = G^n = 1, \quad
        x^n = X^n = 0, \quad
        gx = qxg, \quad
        GX = qXG, \\
     gG = Gg, \quad 
        xG = qGx, \quad
        gX = q^{-1}Xg, \quad
        xX = Xx + G - g.
  \end{gather*}
  
  \vspace{-.28in} \qed
  
\end{lemma}

Note that $X$ is $(1,G)$-skew primitive in $D(T_n(q))$, whereas it is $(G,1)$-skew primitive in $T_n(q)^* \cong T_n(q)$, because $D(T_n(q))$ contains a copy of $T_n(q)^{* cop}$.

\begin{theorem}\cite[Theorem~4.5]{montschneid} \label{thm:taftext}
   Take $n > 2$.
   Let $A = \kk[u]/(u^n - \beta)$ for $0 \neq \beta \in \kk$ be an $n$-dimensional inner-faithful $T_n(q)$-module algebra with no nonzero nilpotent elements, such that $g \cdot u = qu$ and $x \cdot u = \gamma 1_A$ for $0 \neq \gamma \in \kk$.
   Then, by defining $G \cdot u = q^{-1} u$ and $X \cdot u = \gamma^{-1}(q^{-1} - 1) u^2$, we obtain that $A(T_n(q))$ is a $D(T_n(q))$-module algebra.
   Moreover, all $D(T_n(q))$-module algebra structures on $A(T_n(q))$ are of this form.
   \qed
\end{theorem}

The original theorem in \cite{montschneid} has the assumption $n > 1$, not $n > 2$.
We now discuss this disparity.

\subsection{The Sweedler algebra ($n = 2$)} \label{subsect:sweedler}

We begin with the following remark pertaining to Theorem~\ref{thm:taftext} in the case when $n = 2$.
\begin{remark}
   The proof of Theorem~\ref{thm:taftext} in \cite{montschneid} fails for $n = 2$ at the point when one considers the action of $H^{*cop} \subset D(H)$, and applies \cite[Theorem~2.2]{montschneid}. 
   To specify the action of $H^{*cop}$, one uses integers $0 \leq s, t \leq n-1$ with $t(1-s) \equiv 1 \mod n$. 
   It is shown then that $t = n-1$, from which it is concluded that $s = 2$. 
   This is valid if $ n > 2 $. 
   However, for $n = 2$, we get that $s = 0$, and \cite[Theorem~2.2]{montschneid} actually gives us different information than when $n > 2$. 
   We explore here the case when $n = 2$, that is, when $H$ is the Sweedler Hopf algebra, $T_2(-1)$.
\end{remark}

Since Theorem~\ref{thm:taft} applies to the case $n=2$, we know all actions of $T_2(-1)$ on $A(T_2(-1))$ as in Notation~\ref{not:ah}.
As an algebra, $A(T_2(-1)) \cong \kk[u]/(u^2 - 1)$, with the actions given by $g \cdot u = -u$ and $x \cdot u = \gamma 1_A$ for some nonzero $\gamma \in \kk$.
Considering the remark above, we now examine Question~\ref{questions}(c) for $H = T_2(-1)$.

\begin{proposition} \label{prop:sweedlerext}
   Recall the notation of Lemma~\ref{lem:taftdouble} for $n=2$, and thus $q=-1$.
   Fix an action of $T_2(-1)$ on $A(T_2(-1)) = \kk[u]/(u^2 - 1)$ as in Theorem~\ref{thm:taft},
   \[
      g \cdot u = -u, \quad x \cdot u = \gamma 1_A,
   \]
   for some nonzero $\gamma \in \kk$.
   Then, for any $\delta \in \kk$, by defining 
   \[
      G \cdot u = -u, \quad X \cdot u = \delta 1_A,
   \]
   we obtain that $A(T_2(-1))$ is a $D(T_2(-1))$-module algebra. 
   Moreover, all extensions of the action of $T_2(-1)$ on $A(T_2(-1))$ to $D(T_2(-1))$ are of this form.
\end{proposition}

   \begin{proof}
      That $A(T_2(-1))$ is a $D(T_2(-1))$-module algebra with the given action of $G$ and $X$ is easily verified, so we show that all extensions of the action of $T_2(-1)$ on $A(T_2(-1))$ to an action $D(T_2(-1))$ are of this form.
      Fix an action of $T_2(-1)$ on $A(T_2(-1))$.
      That is, we have $A \defeq A(T_2(-1)) = \kk[u]/(u^2-1)$ with the action of $T_2(-1)$ on $A$ given by $g \cdot u = -u$ and $x \cdot u = \gamma 1_A$.
      We can decompose $A$ by the eigenspaces of the action of $g$ as in Remark~\ref{rem:astructure}: $A = A_0 \oplus A_1$ with $A_0 = \kk 1_A$ and $A_1 = \kk u$.
      Now assume this action can be extended to an action of $D(T_2(-1))$.
      Since $A$ is a $D(T_2(-1))$-module algebra, ${g \cdot (G \cdot u) = G \cdot (g \cdot u) = - G \cdot u }$. 
      Hence, $G \cdot u \in A_1$, so $G \cdot u = \alpha u$ for some $ \alpha \in \kk $. 
      Also, we have $x \cdot ( G \cdot u ) = - G \cdot ( x \cdot u ) = - G \cdot \gamma 1_A = - \gamma 1_A$, so $ \alpha = -1.$ 
      Finally, $g \cdot (X \cdot u) = - X \cdot (g \cdot u) = X \cdot u$ implies that $X \cdot u \in A_0 = \kk 1_A$, so $X \cdot u = \delta 1_A$ for some $\delta \in \kk$.
      (Note that $(xX - Xx) \cdot u = (G - g) \cdot u = 0$, so no restrictions on $\delta$ need to be imposed.)
   \end{proof}

All the results about the Taft algebras, including the Sweedler algebra --- Montgomery and Schneider's results (stated in Theorem~\ref{thm:taft} and Theorem~\ref{thm:taftext}) as well as Proposition~\ref{prop:sweedlerext} --- can be realized as a corollary of results about a generalization of Taft algebras, which we consider next.

\section{$H_n(\zeta,m,t)$, a coradically graded generalization of Taft algebras} \label{sect:qls}

We wish to answer Question~\ref{questions} for a family of coradically graded Hopf algebras that contains the Taft algebras. 
In the language of Nichols algebras, $T_n(q)$ is of Cartan type $A_1$ and has rank $1$.
In fact, $T_n(q) \cong \mf B(V) \# \kk \Gamma$, where $(V,c) = \kk x$ is a one-dimensional braided vector space with braiding $c(x \tens x) = q x \tens x$, $\Gamma = \langle g \rangle$ the cyclic group of order $n$, $g \cdot x = qx$, and $\rho(x) = g \tens x$.
That is, $T_n(q)$ is a bosonization of the quantum linear space $\mf B(V)$ of rank 1.
Thus, we consider more generally all rank $1$ quantum linear spaces over finite cyclic groups.

\subsection{The Hopf Algebras $H_n(\zeta, m, t)$} \label{subsect:hnzmt}

We describe all bosonizations of quantum linear spaces of rank 1 over finite cyclic groups.
In our consideration of these, we will need the following result in group theory.

\begin{lemma} \label{lem:numtheory}
   If $G$ is a cyclic group of order $n$, and an element $g \in G$ has order $n/k$ for some $k | n$, then there exists a generator $y$ of $G$ such that $g = y^k$.
   \qed
\end{lemma} 

Let $\Gamma$ be a finite cyclic group of order $n$.       
In the notation of \cite{as-qls}, a quantum linear space of rank 1 over $\Gamma$, denoted $\mc R (g, \chi)$, is entirely determined by a choice of $g \in \Gamma$ and $\chi \in \widehat \Gamma$ such that $\chi(g) \neq 1$.
Fix a non-identity element $g \in \Gamma$.
Then $g$ has order $n/m$ for some $m | n$, and by Lemma~\ref{lem:numtheory}, we can choose a generator $y$ of $\Gamma$ so that $g = y^m$.
Similarly, fix a non-identity element $\chi \in \widehat \Gamma$.
Then $\chi(y)$ is an $n^{th}$ root of unity, say of order $n/t$ with $t | n$, and again by Lemma~\ref{lem:numtheory}, we can choose a primitive $n^{th}$ root of unity, $\zeta$, such that $\chi(y) = \zeta^t$.
We have $ \chi(g) = \zeta^{mt}$, and $N = \ord(\chi(g)) = \frac{n}{\text{gcd}(n, mt)}$.
Our assumption that $\chi(g) \neq 1$ means precisely that $n \nmid mt$.
In this case, $\mc R (g, \chi) $ has a single generator, $x$, and a single relation, $x^N = 0$. 
By definition, $\mc R (g, \chi)$ is a braided Hopf algebra in ${}_\Gamma ^\Gamma \mc{YD}$ with $\rho(x) = g \tens x = y^m \tens x$ and $y \cdot x = \chi(y) x = \zeta^t x$.
Therefore, the bosonization $\mc R(g, \chi) \#\kk \Gamma$ is a Hopf algebra.
The structure of $\mc R(g, \chi) \#\kk \Gamma$ is similar to that of a Taft algebra, as we now describe.

\begin{definitionproposition} \label{def:hnzmt}
	Let $m,t$ be positive integer divisors of $n$ such that $n \nmid mt$ and let $\zeta$ be a primitive $n^{th}$ root of unity.
	Define $H_n(\zeta, m, t)$ as the $\kk$-algebra generated by $y$ and $x$, subject to the relations
	\[
	  y^n = 1, \quad \quad x^N = 0 \text{ for } N = \ord(\zeta^{mt}), \quad \quad
	    yx = \zeta^t xy.
	\]
   The algebra $H_n(\zeta, m, t)$ has a unique Hopf algebra structure determined by 
   \begin{gather*}
      \Delta(y) = y \tens y, \quad
      \Delta(x) = y^m \tens x + x \tens 1, \quad 
      \epsilon(y) = 1, \quad
      \epsilon(x) = 0, \quad
      S(y) = y^{-1}, \quad
      S(x) = -y^{-m} x.
   \end{gather*}
   Here, $H_n(\zeta, m, t) \cong \mc R(g, \chi) \#\kk\Gamma$, where $g = y^m$ and $\chi(y) = \zeta^t$.
   Such Hopf algebras have dimension $Nn$. \qed
\end{definitionproposition}

For a fixed $n$, a natural first question is whether each choice of $\zeta$, $m$, and $t$ determines a unique Hopf algebra.
Unsurprisingly, the answer is negative; however, an isomorphism class does uniquely determine $n$, $m$, and $t$.
To show this, we require the following lemma characterizing certain primitive elements.

\begin{lemma}{\cite[Corollary~5.3]{as-qls}} \label{lem:primitives}
   Let $0 \leq b < n$.
   Then 
   \[
      P_{y^b, 1}(H_n(\zeta, m, t)) = \begin{cases}
        \kk x +\kk(y^b - 1), & \text{ if } b \equiv m \mod n \\
        \kk (y^b - 1), & \text{ otherwise}.
      \end{cases}
   \]
   
   \vspace{-.25in} \qed
\end{lemma} 

\begin{proposition}
   Let $m, \wh m, t, \wh t$ be positive divisors of $n$ such that $n$ divides neither $mt$ nor $\wh m \wh t$.
   Let $\zeta, \wh \zeta$ be primitive $n^{th}$ roots of unity in $\kk$. 
   Then $H_n(\zeta, m, t) \cong H_{\wh n}(\wh \zeta, \wh m, \wh t)$ if and only if $n = \wh n$, $m = \wh m$, $t = \wh t$, and there exists $ f \in (\Z / n\Z)^\times$ such that $(\wh \zeta)^{ft} = \zeta^t$ and $fm \equiv m \mod n$.
   As a consequence, for fixed $n \in \N$ and a fixed primitive $n^{th}$ root of unity $\zeta$, each choice of $m,t \in \N$ with both dividing $n$ and $n \nmid mt$ yields a unique isomorphism class of Hopf algebras $H_n(\zeta,m,t)$.
\end{proposition}

   \begin{proof}
      Let $y,x$ denote the generators of $H_n(\zeta, m, t)$, and $\hat y, \hat x$ the generators of $H_{\wh n}(\wh \zeta, \wh m, \wh t)$.
      Assume the conditions on $\wh n$, $\wh m$, $\wh t$, and $f$.
      The isomorphism between the two is defined by sending $y$ to $\wh y^f$ and $x$ to $\wh x$.
      One can easily check that this defines a Hopf algebra isomorphism.
      
      On the other hand, suppose $H_n(\zeta, m, t)$ and $H_{\wh n}(\wh \zeta, \wh m, \wh t)$ are isomorphic and let $\phi$ denote an isomorphism between them.
      By counting grouplike elements, $n = \wh n$.
      Moreover, $\phi(y)$ must be a grouplike element of order $n$.
      Thus, there exists $f \in (\Z/n\Z)^\times$ such that $\phi(y) = \wh y ^f$.
      Since $(\phi \tens \phi) \circ \Delta = \Delta \circ \phi$, we must have $\phi(x) \in P_{\wh y^{fm}, 1}(H_{\wh n}(\wh \zeta, \wh m, \wh t))$.
      By Lemma~\ref{lem:primitives}, 
      \[
         P_{\wh y^{fm}, 1}(H_{\wh n}(\wh \zeta, \wh m, \wh t)) = \begin{cases}
           \kk \wh x +\kk (\wh y^{fm} - 1), & fm \equiv \wh m \mod n \\
           \kk (\wh y^{fm} - 1), & \text{ otherwise.}
         \end{cases}
      \]
      Since $\phi(x)$ and $\wh y^f$ must generate $H_{\wh n}(\wh \zeta, \wh m, \wh t)$, it must be that $fm \equiv \wh m \mod n$ and $\phi(x) = \alpha \wh x + \beta (\wh y^{fm} - 1)$ for some $\alpha, \beta \in \kk$ with $\alpha \neq 0$.
      Now, since $m$ and $\wh m$ both divide $n$, and $f$ is a unit mod $n$, the equation $fm \equiv \wh m \mod n$ implies $m = \wh m$.
      We must have 
      \begin{align*}
         0 &= \phi(yx - \zeta^t xy) 
         ~=~ \wh y^f (\alpha \wh x + \beta ( \wh y^{fm} - 1)) - \zeta^t (\alpha \wh x + \beta ( \wh y^{fm} - 1)) \wh y^f \\
         &= ((\wh \zeta)^{f \wh t} - \zeta^t) \alpha \wh x \wh y^f + (1 - \zeta^t) \beta \wh y^f (\wh y^{fm} - 1).
      \end{align*}
      Thus, since $\zeta^t \neq 1$ and $\wh y^{fm} \neq 1 \  (\text{as } n \nmid mt$), we must have $\beta = 0$.
      Also, since $\alpha \neq 0$, we have $(\wh \zeta)^{f \wh t} = \zeta^t$.
      Since $\zeta$ and $\wh \zeta$ are primitive $n^{th}$ roots of unity, $\wh \zeta = \zeta^e$ for some $e \in (\Z / n\Z)^\times$.
      Therefore, $ ef\wh t \equiv t \mod n$, and just as for $m=\wh m$, we see that $t = \wh t$.
   \end{proof}   
   
Not only do the Hopf algebras just presented include Taft algebras; they also include the coradically graded generalized Taft algebras.

\begin{definition} \label{def:gentaft}
  For natural numbers $n,N$ satisfying $N \mid n$, a primitive $N^{th}$ root of unity $q \in \kk$, and $\alpha \in \kk$ arbitrary, the \emph{generalized Taft algebra} $T(n,N,\alpha)$ is the Hopf algebra generated by a grouplike element $g$ and a $(g,1)$-skew primitive element $x$, subject to the relations
   \[
      g^n = 1, \quad x^N = \alpha (g^N - 1), \quad gx = qxg.
   \]
\end{definition}
   
\begin{proposition} \label{prop:gtaftasqls}
   The Hopf algebra $H_n(\zeta, m, t)$ is isomorphic to a generalized Taft algebra of the form $T(n, N, 0)$ if and only if $m=1$. 
   In this case, $q = \zeta^t$, and $N = n/t$.
   Moreover, any generalized Taft algebra of the form $T(n, N, 0)$ can be realized as such.
\end{proposition}
 
   \begin{proof}
      Assume $m = 1$.
      Then $N = n/t$ and $q = \zeta^t$ is a primitive $N^{th}$ root of unity by definition.
      Now, let $g = y$ and note that $x$ is $(g,1)$-skew primitive.
      One easily checks that $g$ and $x$ satisfy all the relations of $T(n, N, 0)$, so by a dimension count, the two are isomorphic.
      
      On the other hand, suppose $\phi: T(n, N, 0) \to H_{n}(\zeta, m, t)$ is an isomorphism.
      By considering the groups of grouplike elements, there exists $e \in (\Z / n\Z)^\times$ such that $\phi(g) = y^e$.
      Since $\phi$ is a map of coalgebras, $\phi(x)$ must be $(y^e, 1)$-skew primitive.
      By Lemma~\ref{lem:primitives}, 
      \[
         P_{y^e, 1}(H_{n}(\zeta, m, t)) = \begin{cases}
            \kk x + \kk (y^e - 1), & \text{if } e \equiv m \mod n \\
            \kk (y^e - 1), & \text{otherwise}.
         \end{cases}
      \]
      Since $\phi(g)$ and $\phi(x)$ must generate $H_{n}(\zeta, m, t)$, we must have $m \equiv e \mod n$.
      Thus, since the only unit mod $n$ that divides $n$ is 1, we get $m = 1$.
      
      Now, let $T(n, N, 0)$ be a generalized Taft algebra. 
      By definition, $N$ divides $n$, and $T(n, N, 0)$ is generated by a grouplike element $g$ and a $(g, 1)$-skew primitive element $x$ subject to the relations $g^n = 1$, $x^N = 0$, and $gx = qxg$ for some primitive $N^{th}$ root of unity $q$.
      Let $t = n/N$ and choose a primitive $n^{th}$ root of unity $\zeta$ such that $\zeta^t = q$.
      It is now easy to see that $T(n, N, 0)$ is precisely $H_n(\zeta, 1, t)$.
   \end{proof}
   
Now the following consequence is clear.
   
\begin{corollary} \label{cor:taftasqls}
   The Hopf algebra $H_n(\zeta, m, t)$ is isomorphic to a Taft algebra if and only if $m=t=1$. 
   In that case, $H_n(\zeta, m, t) \cong T_n(\zeta)$.
   If, further, $n$ is prime, then every Hopf algebra of the form $H_n(\zeta,m,t)$ is a Taft algebra.
   \qed
\end{corollary}
   
A consequence of Lemma~\ref{lem:primitives} and Corollary~\ref{cor:ifprims} is the following:

\begin{corollary} \label{cor:hnzmtinnerfaithful}
   A left $H_n(\zeta, m, t)$-module $M$ is inner-faithful if and only if $G(H_n(\zeta, m, t)) = \langle y \rangle$ acts faithfully on $M$ and $x \cdot M \neq 0$.
\end{corollary}

   \begin{proof}
      The forward direction is clear.
      Assume, then, that $\langle y \rangle$ acts faithfully and that $x \cdot M \neq 0$.
      Then every nonzero multiple of $y^b - 1$ does not act by zero for every $b$.
      Thus, we only need to check that each nonzero element of $\kk x + \kk (y^m - 1)$ acts by nonzero by Corollary~\ref{cor:ifprims} and Lemma~\ref{lem:primitives}.
      Since $x$ and $y^m - 1$ do not act by zero, this is equivalent to showing that $x$ does not act as any nonzero scalar multiple of $y^m - 1$.
      Let $M_i = \{a \in M : y \cdot a = \zeta^i a \}$ denote the eigenspaces of the action of $y$ on $M$. 
      Note that if $u \in M_i$, then $x \cdot u \in M_{i+t}$, since $y \cdot (x \cdot u) = \zeta^t x \cdot (y \cdot u) = \zeta^{i+t} x \cdot u$.
      If $x \cdot M_i = 0$ for all $i$, then $x \cdot M = 0$, contradicting our hypothesis.
      Thus, choose $i$ and $u \in M_i$ such that $x \cdot u \neq 0$.
      Then $(y^m - 1) \cdot u = (\zeta^{mi} - 1) u \in M_i$,~but $x \cdot u \in M_{i+t}$.
      Since $n \nmid mt$, $M_i \neq M_{i+t}$.
      Thus, $x \cdot u$ is not equal to any nonzero scalar multiple of $(y^m - 1) \cdot u$.
   \end{proof}
   
Our next goal is to answer Question~\ref{questions} for $H_n(\zeta,m,t)$.
That is, we are interested in the existence of structures $A(H_n(\zeta, m, t))$ as in Notation~\ref{not:ah}, and whether or not such structures can be extended to admit actions of $D(H_n(\zeta, m, t))$.
Before considering this, we compute $D(H_n(\zeta, m, t))$ explicitly. 
This is made easier by first giving a nice presentation of the dual.

\subsection{Computing the dual and double}

\subsubsection{The dual $H_n(\zeta, m, t)^*$} \label{subsubsect:hnzmtdual}

In \cite{beattie}, Beattie computed the duals of quantum linear spaces.
As an application of \cite[Corollary~2.3]{beattie}, we get the following result:

\begin{lemma} \cite{beattie}
   As Hopf algebras, $H_n(\zeta, m, t)^* \cong H_n(\zeta, t, m)$. \qed
\end{lemma}

Since we have a presentation of the dual, for computing the double, we would like to know the dual pairing between $H_n(\zeta,m,t)$ and ${H_n(\zeta,t,m)}$.
Thus, we exhibit a perfect duality between these two Hopf algebras.

\begin{proposition} \label{prop:hnzmtpairing}
  With $y,x$ denoting the generators of $H_n(\zeta,m,t)$,  and $Y,X$ the generators $H_n(\zeta,t,m)$, the bilinear form defined by
  \begin{equation} \label{eq:hnzmtpairing}
     \langle X^i Y^j, x^k y^\ell \rangle = \delta_{i,k} \ (i)_q ! \ \zeta^{j \ell},
  \end{equation}
  is a perfect duality.
\end{proposition}

In particular, we get that the dual pairing is given on generators by
  \[
     \langle Y, y \rangle = \zeta, \quad%
     \langle Y, x \rangle = 0, \quad%
     \langle X, y \rangle = 0, \quad%
     \langle X, x \rangle = 1.
  \]
Note the following equalities, which will be useful for our calculations:
\begin{gather}
   \Delta(X^i Y^j) = \sum_{s=0}^i {\binom i s}_q X^{i-s} Y^{ts + j} \tens X^s Y^j \quad \text{and} \label{eq:hnzmtcomult} \\
   S(x^i y^j) ~=~ S(y^j)S(x^i) ~=~ (-1)^i y^{-j} q^{i-1} y^{-im} x^i ~=~ (-1)^i q^{i-1} \zeta^{-ti(im+j)} x^i y^{-im - j} . \label{eq:hnzmtantipode}
\end{gather}

  \begin{proof}[Proof of Proposition~\ref{prop:hnzmtpairing}]
    We show that \eqref{eq:hnzmtpairing} is a duality, i.e. that \eqref{eq:duality} holds.
    First, we check that 
    \begin{equation} \label{eq:hnzmtduala}
      \langle X^a Y^b, x^i y^j x^k y^\ell \rangle  =  \langle \com{(X^a Y^b)} 1, x^i y^j \rangle \langle \com{(X^a Y^b)} 2, x^k y^\ell \rangle.
    \end{equation}
    On the one hand,
    \begin{align*}
       \langle X^a Y^b, x^i y^j x^k y^\ell \rangle 
          &= \zeta^{tjk} \langle X^a Y^b, x^{i+k} y^{j+\ell} \rangle 
          \overset{\eqref{eq:hnzmtpairing}} = \delta_{a, i+k} \ (a)_q ! \ \zeta^{tjk + b(j + \ell)}.
    \end{align*}
    On the other hand, we have 
    \begin{align*}
       \langle \com{(X^a Y^b)} 1, x^i y^j \rangle \langle \com{(X^a Y^b)} 2, x^k y^\ell \rangle
          & \overset{\eqref{eq:hnzmtcomult}} = \sum_{s = 0}^a {\binom a s}_q \langle X^{a-s} Y^{ts + b}, x^i y^j \rangle \langle X^s Y^b, x^k y^\ell \rangle \\
          & \overset{\eqref{eq:hnzmtpairing}} = \sum_{s = 0}^a {\binom a s}_q \delta_{a-s, i} \ (a-s)_q ! \ \zeta^{(ts + b) j} \ \delta_{s,k} \ (s)_q ! \ \zeta^{b \ell} \\
          & \overset{\ \ \ \ \ \ } = \delta_{a, i+k} {\binom a k}_q \ (i)_q ! \ (k)_q ! \ \zeta^{tjk + bj + b \ell}.
    \end{align*}
    Now \eqref{eq:hnzmtduala} follows since when $a = i+k$, we have ${\binom a k}_q = \frac {(a)_q!}{(i)_q!(k)_q!}$.
    
    The proof that $ \langle X^a Y^b X^c Y^d, x^i y^j \rangle = \langle X^a Y^b, \com{(x^i y^j)} 1 \rangle \langle X^c Y^d, \com{(x^i y^j)} 2 \rangle $
    follows similarly.
    We also have by \eqref{eq:hnzmtpairing} that
    $
       \langle X^a Y^b, 1 \rangle = \delta_{a,0} = \epsilon(X^a Y^b) 
       \text{ and }
       \langle 1, x^i y^j \rangle = \delta_{0,i} = \epsilon(x^i y^j).
    $
    
    Finally, we have that
    \begin{align*}
       \langle X^a Y^b, S(x^i y^j) \rangle
          & \overset{\eqref{eq:hnzmtantipode}} = (-1)^i q^{i-1} \zeta^{-ti(im+j)} \langle X^a Y^b, x^i y^{- im - j} \rangle
          \overset{\eqref{eq:hnzmtpairing}} = \delta_{a, i} \ (a)_q! \ (-1)^i q^{i-1} \zeta^{-ti(im+j)} \ \zeta^{-b (im + j)} \\
          & \overset{\ \ \ \ \ \ } = \delta_{a, i} \ (a)_q! \ (-1)^a q^{a-1} \zeta^{-ma(at + b)} \ \zeta^{-j (at + b)} 
          \overset{\eqref{eq:hnzmtpairing}} = (-1)^a q^{a-1} \zeta^{-ma(at + b)} \langle X^a Y^{-at - b}, x^i y^j \rangle \\
          & \overset{\eqref{eq:hnzmtantipode}} = \langle S(X^a Y^b), x^i y^j \rangle.
    \end{align*}
    Therefore, we have a duality. 
    To show that this duality is perfect, we need to show that the maps \mbox{$\phi: H_n(\zeta, t, m) \to H_n(\zeta, m, t)^*$} and $\psi: H_n(\zeta, m, t) \to H_n(\zeta, t, m)^*$ defined by
    $
        \phi(u)(x) = \langle u, x \rangle = \psi(x)(u)
    $
    are injective.
    By a dimension count, verifying just one of these claims suffices.
    Let $f = \sum_{a = 0}^{N-1} \sum_{b = 0}^{n-1} \alpha_{a,b} X^a Y^b$ with $\alpha_{a,b} \in \kk$ and suppose $\phi(f) = 0$.
    Then for any $i,j$,
    \begin{align*}
       0 &= \phi(f)(x^i y^j) 
          = \langle f, x^i y^j \rangle
          = \sum_{a = 0}^{N-1} \sum_{b=0}^{n-1} \alpha_{a,b} \langle X^a Y^b, x^i y^j \rangle 
          = \sum_{a = 0}^{N-1} \sum_{b=0}^{n-1} \alpha_{a,b} \; \delta_{a,i}\; (a)_q! \; \zeta^{bj}
          = \sum_{b=0}^{n-1} \alpha_{i,b} \; (i)_q! \; \zeta^{bj} .
    \end{align*}
    Let $\beta_{i,j}$ denote $\sum_{b = 0}^{n-1} \alpha_{i,b} \; \zeta^{bj}$.
    By the above, for every $i,j$, $\beta_{i,j} = 0$.
    Thus, for any fixed $i$ and $k$,
    \[
       0 = \sum_{j=0}^{n-1} \zeta^{-jk} \beta_{i,j} 
       = \sum_{j=0}^{n-1} \zeta^{-jk} \sum_{b=0}^{n-1} \alpha_{i,b} \; \zeta^{bj} 
       = \sum_{b=0}^{n-1} \left( \sum_{j=0}^{n-1} \zeta^{(b-k)j} \right) \alpha_{i,b} 
       = n \ \alpha_{i,k} .
    \]
    The last equality follows because for $\xi$ a non-identity $n^{th}$ root of unity, $\sum_{j=0}^{n-1} \xi^j = 0$, and $\zeta^{b-k} \neq 1$ for all $b \neq k$.
    Thus, since each $\alpha_{i,j} = 0$, we have $f = 0$, so $\phi$ is injective.
    Hence, we have proven that the duality is perfect.    
  \end{proof}

\subsubsection{The Drinfel'd double $D(H_n(\zeta, m, t))$} \label{subsubsect:hnzmtdouble}

Now we begin with the computation of $D(H_n(\zeta, m, t))$.
By Lemma~\ref{lem:gens}, as an algebra, $D(H_n(\zeta, m, t))$ is generated by the generators of $H_n(\zeta, m, t)$ and of its dual, and has the relations of both.
We only need to find how these generators ``commute'' with each other, i.e. how to in general write an element as a linear combination of monomials with $X$ and $Y$ to the left of $x$ and $y$.
To find these relations, we use \eqref{eq:doublemixed}.

\begin{proposition}
  The Drinfel'd double $D(H_n(\zeta, m, t))$ of $H_n(\zeta,m,t)$ is generated by grouplike elements $y$ and $Y$, a $(y^m, 1)$-skew primitive element $x$, and a $(1, Y^t)$-skew primitive element $X$, subject to the relations 
  \begin{gather*}
     y^n = Y^n = 1, \quad 
     x^N = X^N = 0, \quad
     yx = \zeta^t xy, \quad
     YX = \zeta^m XY, \\
     yY = Yy, \quad
     xY = \zeta^m Yx, \quad
     yX = \zeta^{-t} Xy, \quad
     xX - Xx = Y^t - y^m,
  \end{gather*}
  where $ N = \ord(\zeta^{mt}) = \frac n {\text{gcd}(n, mt)}$. 
\end{proposition}

   \begin{proof}
      The generators and first row of relations follow from Lemma~\ref{lem:gens}.
      The remaining relations come from moving generators of one across generators of the other, which is done as follows.
      First, note that 
      \begin{gather*}
         \Delta^2(x) = y^m \tens y^m \tens x + y^m \tens x \tens 1 + x \tens 1 \tens 1, \quad
         \Delta^2(y) = y \tens y \tens y, \\
         \Delta^2(X) = Y^t \tens Y^t \tens X + Y^t \tens X \tens \epsilon + X \tens \epsilon \tens \epsilon,
         \quad \Delta^2(Y) = Y \tens Y \tens Y,
      \end{gather*}
      and that $S^{-1}(x) = -xy^{-m}$.
      Thus, using \eqref{eq:doublemixed} and \eqref{eq:hnzmtpairing}, we have the following computations
      \[
         y Y = \langle Y, y^{-1} \rangle \langle Y, y \rangle Y y 
            = \zeta^{-1} \zeta Y y 
            = Y y ,
      \]
      \[
         x Y ~=~ \langle Y, -xy^{-m} \rangle \langle Y, y^m \rangle Y y^m 
               + \langle Y, 1 \rangle \langle Y, y^m \rangle Y x,
               + \langle Y, 1 \rangle \langle Y, x \rangle Y 1
            ~=~ \zeta^m Y x,
      \]
      \[
         y X ~=~ \langle Y^t, y^{-1} \rangle \langle X, y \rangle Y^t y
               + \langle Y^t, y^{-1} \rangle \langle \epsilon, y \rangle X y
               + \langle X, y^{-1} \rangle \langle \epsilon, y \rangle \epsilon y
            ~=~ \zeta^{-t} X y, \quad \text{and}
      \]
      \begin{alignat*}{3}
         x X ~=~& \langle Y^t, -xy^{-m} \rangle \langle X, y^m \rangle Y^t y^m
               &&+ \langle Y^t, -xy^{-m} \rangle \langle \epsilon, y^m \rangle X y^m
               &&+ \langle X, -xy^{-m} \rangle \langle \epsilon, y^m \rangle \epsilon y^m  \\
             &\ \ + \langle Y^t, 1 \rangle \langle X, y^m \rangle Y^t x
               &&+ \langle Y^t, 1 \rangle \langle \epsilon, y^m \rangle X x
               &&+ \langle X, 1 \rangle \langle \epsilon, y^m \rangle \epsilon x  \\
             &\ \ + \langle Y^t, 1 \rangle \langle X, x \rangle Y^t 1
               &&+ \langle Y^t, 1 \rangle \langle \epsilon, x \rangle X 1
               &&+ \langle X, 1 \rangle \langle \epsilon, x \rangle \epsilon 1  \\
            ~=~& -y^m + X x + Y^t.
      \end{alignat*}
      
      \vspace{-.28in}
      
   \end{proof}

\subsection{The possible structures of $A(H_n(\zeta, m, t))$}

We will see that $H_n(\zeta, m, t)$-module algebra structures on $A(H_n(\zeta, m, t))$ as in Notation~\ref{not:ah} do not always exist, depending on the value of $m$ and $t$.
For considering actions of $H_n(\zeta, m, t)$ on $A(H_n(\zeta, m, t))$, we will use an infinite-dimensional Hopf algebra for which $H_n(\zeta, m, t)$ is a factor.
For an integer $n>0$, a primitive $n^{th}$ root of unity $\zeta \in \kk$, and $m, t \in \Z$ both dividing $n$, we define
\[
  \widetilde H_n(\zeta, m, t) = \kk \langle y,  x \mid y^n = 1, \ y  x = \zeta^t  x y \rangle,
\]
with $y$ grouplike, and $ x$ a $(y^m,1)$-skew primitive element.
It is clear that $H_n(\zeta, m, t)$ is the factor of $\widetilde H_n(\zeta, m, t)$ by the Hopf ideal generated by $ x^N$.
The following technical lemma will help us determine when structures as in Notation~\ref{not:ah} do exist.
We will see that the obstruction comes from the condition that $x^N$ acts by zero.
We again use the notation from Remark~\ref{rem:astructure} for eigenspaces of the action of $y$: $A_i = \{a \in A \mid y \cdot a = \zeta^i a\}$, noting that by a dimension count $A_i = \kk u^i$ for all $i$.

\begin{lemma} \label{lem:htildema}
   Let $A = \kk[u] / (u^n -1)$ and suppose $A$ is an $\widetilde H_n(\zeta, m, t)$-module algebra with $y \cdot u = \zeta u$ and $ x \cdot u \neq 0$.
   Then there exists nonzero $\gamma \in \kk$ such that for any $p, q > 0$,
   \begin{gather*}
       x \cdot u^p = \gamma \ (p)_{\zeta^m}\  u^{p+t} \quad \text{and} \quad
       x^q \cdot u^p = \gamma^q \left(\prod_{i=0}^{q-1} (p + it)_{\zeta^m} \right) u^{p + qt}.
   \end{gather*}
   In particular, $ x^N \cdot u^p = 0$ if and only if $n/m$ divides $p + it$ for some $0 \leq i < N$.
\end{lemma}

   \begin{proof}
      First, since $y  x \cdot u = \zeta^t  x y \cdot u = \zeta^{t+1}  x \cdot u$, we see that $ x \cdot u \in A_{t+1} = \kk u^{t+1}$.
      Thus, there exists nonzero $\gamma \in \kk$ such that $ x \cdot u = \gamma u^{1+t}$.
      We have established the first equality for the case $p=1$.
      Thus, we proceed by induction, assuming the result for $p-1$.
      We compute:
      
      \begin{tabular}{r l l}
         $ x \cdot u^p $
           &$= (y^m \cdot u)( x \cdot u^{p-1}) + ( x \cdot u)(1 \cdot u^{p-1})$
           &$=  (\zeta^m u)(\gamma \ (p-1)_{\zeta^m} \ u^{p-1 + t}) + (\gamma u^{t+1})(u^{p-1})$ \\
           &$= \gamma \ [\zeta^m \ (p-1)_{\zeta^m} \ + 1] \ u^{p+t} $
           &$= \gamma \ (p)_{\zeta^m}\  u^{p+t}$.
      \end{tabular}
      
      This establishes the first result for all $p$, as well as the second equality in the case $q=1$.
      We now prove the second equality for all $q$ and $p$, by induction on $q$.
      Assume the result for $q-1$.
      Then we compute:
      \begin{align*}
          x^q \cdot u^p =  x \cdot ( x^{q-1} \cdot u^p)
         &=  x \cdot \left( \gamma^{q-1} \ \left( \prod_{i=0}^{q-2} (p + it)_{\zeta^m} \right) \ u^{p + (q-1)t} \right) \\
         &= \gamma^{q-1} \ \left( \prod_{i=0}^{q-2} (p + it)_{\zeta^m} \right) \gamma \ (p + (q-1)t)_{\zeta^m} \ u^{p + (q-1)t + t} 
         = \gamma^q \left( \prod_{i=0}^{q-1} (p + it)_{\zeta^m} \right) u^{p + qt}.
      \end{align*}
      The final statement holds as $(n)_q = 0$ if and only if $\ord(q) \mid n$, and as $\ord(\zeta^m) = n/m$.
   \end{proof}

\begin{proposition} \label{prop:minexistence}
   There exist $H_n(\zeta, m, t)$-module algebra structures on $A(H_n(\zeta, m, t))$ as in Notation~\ref{not:ah} if and only if one of the following equivalent conditions holds:
   \begin{enum}
      \item
         ${\text{gcd}(t, n/m) = 1}$
      \item
         $\text{gcd}(mt, n) = m$
      \item
         $n/m = N \ (= \ord(\zeta^{mt}))$ 
   \end{enum}
   In particular, if $t = 1$, then there are $H_n(\zeta,m,t)$module algebra structures on $A(H_n(\zeta,m,t))$ as in Notation~\ref{not:ah}.
   On the other hand, if these structures exist, we must have that $t | m$, and in this case, the module structure is given by $y \cdot u = \zeta u$ and $x \cdot u = \gamma u^{t+1} $ for some nonzero $\gamma \in \kk$.
   
\end{proposition}

   \begin{proof}
      The equivalence of the three conditions follows from elementary group theory and number theory.
      First, assume these conditions hold.
      By definition, $A(H_n(\zeta,m,t)) = \kk[u]/(u^n - 1)$.
      For any nonzero $\gamma \in \kk$, by defining $y \cdot u = \zeta u$ and $ x \cdot u = \gamma u^{1+t}$, it is easy to check that $A(H_n(\zeta,m,t))$ is a $\widetilde H_n(\zeta,m,t)$-module algebra.
      In order to get a $H_n(\zeta,m,t)$-module algebra structure, we need only check that $x^N$ acts by zero.
      By Lemma~\ref{lem:htildema}, we must check that for each $p$, we get that $n/m$ divides $p + it$ for some $0 \leq i < N$.
      By assumption, $n/m = N$ is relatively prime to $t$.
      Thus, for any value of $p$, $\{p + it\}_{i = 0}^{N-1}$ consists of $N$ distinct values mod $N$.
      Thus, for exactly one value of $i$, we have $p + it \equiv 0 \mod N$.
      Therefore, $x^N \cdot u^p = 0$ for all $p$, so we have an $H_n(\zeta,m,t)$-module algebra structure.
      By Corollary~\ref{cor:hnzmtinnerfaithful}, this action is inner-faithful.
      
      On the other hand, fix an $H_n(\zeta, m, t)$-module algebra structure $A \defeq A(H_n(\zeta, m, t)) \cong \kk[u]/(u^n - 1)$.
      Since the $H_n(\zeta,m,t)$-module structure on $A(H_n(\zeta, m, t))$ is inner-faithful, by Corollary~\ref{cor:hnzmtinnerfaithful}, $x \cdot u \neq 0$.
      By pulling back along the projection $\widetilde H_n(\zeta,m,t) \to H_n(\zeta,m,t)$, $A$ is a $\widetilde H_n(\zeta,m,t)$-module algebra, with $ x^N \cdot u = 0$.
      Thus, by Lemma~\ref{lem:htildema}, we have $x \cdot u = \gamma u^{t+1}$.
      Moreover, by the same lemma, $1 + it \equiv 0 \mod {n/m}$ for some $0 \leq i < N$.
      That is, we can write $1 = -it + bn/m$ for some $i,b \in \Z$.
      Therefore, $\text{gcd}(t, n/m) = 1$.
   \end{proof}

Proposition~\ref{prop:minexistence} generalizes Montgomery and Schneider's result (stated in Theorem~\ref{thm:taft}), which examines the Taft algebras (the case that $m = t = 1$).
Note that in their work, $x$ acts by lowering the degree of $u$ rather than raising it.
This is due to the fact that they use the relation $xy = \zeta yx$ rather than $yx = \zeta xy$.   
By Proposition~\ref{prop:gtaftasqls} and Corollary~\ref{cor:taftasqls}, we obtain the following result for coradically graded generalized Taft algebras, in general.
   
\begin{corollary}
   Consider a coradically graded generalized Taft algebra $T(n, N, 0) = H_n(\zeta, 1, n/N)$ for some $N$ dividing $n$.
   Then $T(n,N,0)$-module algebra structures on $A(T(n,N,0))$ as in Notation~\ref{not:ah} exist if and only if $n = N$, i.e. if and only if $T(n,N,0)$ is a Taft algebra.
   \qed
\end{corollary}

Thus, we have answered Question~\ref{questions}(a,b) for coradically graded generalized Taft algebras.
We will consider a non-coradically graded generalized Taft algebra, namely $T(4,2,1)$, in Section~\ref{sect:gentaft}.

\subsection{Extending module algebra structures on $A(H_n(\zeta,m,t))$ to $D(H_n(\zeta,m,t))$}

Recall that the Hopf algebra $H_n(\zeta,m,t)$ is determined by a primitive $n^{th}$ root of unity $\zeta$ in $\kk$ and two positive integer divisors of $n$: $m$, which is used to define the coalgebra structure, and $t$ which is used to define the algebra structure.
It is also assumed that $n \nmid mt$.
We now assume $H_n(\zeta,m,t)$-module algebra structures on $A(H_n(\zeta, m, t))$ as in Notation~\ref{not:ah} exist (that is, that $\gcd(t,n/m) = 1$, by Proposition~\ref{prop:minexistence}) and explore when such structures extend to be $D(H_n(\zeta, m, t))$-module algebras.
Recall from Section~\ref{subsubsect:hnzmtdouble} that $D(H_n(\zeta, m, t))$ is generated by grouplike elements $y$ and $Y$, a $(y^m, 1)$-skew primitive element $x$, and a $(1, Y^t)$-skew primitive element $X$, subject to the relations 
\begin{gather*}
   y^n = Y^n = 1, \quad
   x^N = X^N = 0, \quad
   yx = \zeta^t xy, \quad
   YX = \zeta^m XY, \\
   yY = Yy, \quad
   xY = \zeta^m Yx, \quad
   yX = \zeta^{-t} Xy, \quad
   xX - Xx = Y^t - y^m
\end{gather*}
where $ N = \ord (\zeta^{mt}) = n/m$. 

\begin{theorem} \label{thm:ahnzmt}
   Fix an $H_n(\zeta,m,t)$-module algebra structure on $A \defeq A(H_n(\zeta, m, t)) = \kk[u]/(u^n - 1)$ as in Notation~\ref{not:ah}. 
   If the action of $H_n(\zeta, m, t)$ extends to make $A$ a $D(H_n(\zeta, m, t))$-module algebra, then there exists a nonzero scalar $\gamma$ and scalar $\delta \in \kk$, and a natural number $ 0 < d < n $ with $m \equiv -dt \mod n$ such that:
   \[
     y \cdot u = \zeta u, \quad 
     Y \cdot u = \zeta^d u, \quad
     x \cdot u = \gamma u^{1+t}, \quad
     \text{and} \quad X \cdot u = \delta u^{1-t}.
   \]
   
   If $m \neq n/2$ (that is, if $N \neq 2$), then $\gamma$ and $\delta$ are related by the identity
   \[
      \gamma \delta = \frac {\zeta^{-m} - 1}{(n-t)_{\zeta^m}}.
   \]
   In this case, the action of $X$ on $A$ is determined by the $H_n(\zeta, m, t)$-module algebra structure, and if further, $t = 1$, then the action of $Y$ is as well.
   
   On the other hand, if $m = n/2$, there is no such equation relating $\gamma$ and $\delta$.
   
   Conversely, the conditions imposed above on $\delta$ and $d$ are sufficient to define a $D(H_n(\zeta, m, t))$-module algebra structure on $A$.
\end{theorem}

We will need the following lemma about $q$-symbols in the proof of Theorem~\ref{thm:ahnzmt}.
It follows from the definitions.

\begin{lemma} \label{lem:quantum} 
   Let $q \neq 1 \in \kk$. 
   Then the following statements hold.
          
   \begin{enum}
  
      \item
         Suppose that $\ord(q) = n$ and $p \equiv r \mod n$ for integers $p, r > 0$.
         Then $(p)_q = (r)_q$;
         
      \item
         If $\ord(q) | m$ and $0 \leq p \leq m$, then $(p)_{q^{-1}} = -q (m-p)_q$.
      
   \end{enum}
   
   \vspace{-.2in}\qed
\end{lemma} 

   \begin{proof}[Proof of Theorem~\ref{thm:ahnzmt}]
      Actions of $H_n(\zeta,m,t)$ on $A(H_n(\zeta,m,t))$ as in Notation~\ref{not:ah} are given by Proposition~\ref{prop:minexistence}.
      In particular, $y \cdot u = \zeta u$, and $x \cdot u = \gamma u^{1+t}$ for some nonzero $\gamma \in \kk$.
      Since $yY \cdot u = Yy \cdot u = \zeta Y \cdot u$, we see that $Y \cdot u \in A_1 = \kk u$.
      Thus, $Y \cdot u = \delta u$ for some $\delta \in \kk$. 
      However, because $Y^n$ must act by the identity, $\delta$ must be an $n^{th}$ root of unity. 
      That is, $\delta = \zeta^d$ for some $0 \leq d < n$.
      
      On one hand, $xY \cdot u = \zeta^d x \cdot u = \zeta^d \gamma u^{1+t}$.
      On the other hand, $\zeta^m Yx \cdot u = \zeta^m \gamma Y \cdot u^{1+t} = \zeta^{m + d(1+t)} \gamma u^{1+t}$. 
      Therefore, since $\gamma \neq 0$, it must be the case that $d \equiv m + d(1+t) \mod n$.
      That is, $m \equiv -dt \mod n$.
      In particular, this implies $d \neq 0$.
      
      We also have $yX \cdot u = \zeta^{-t}Xy \cdot u = \zeta^{1-t} X \cdot u$, showing that $X \cdot u \in A_{1-t} = \kk u^{1-t}$. 
      Thus, $X \cdot u = \delta u^{1-t}$ for some $\delta \in \kk$.
      One sees by induction that $X \cdot u^p = \delta~(p)_{\zeta^{dt}} u^{p-t}$.  
      Thus, on one hand, by Lemma~\ref{lem:htildema} and Lemma~\ref{lem:quantum},
      \[
         (xX - Xx) \cdot u = \delta x \cdot u^{1-t} - \gamma X \cdot u^{1+t} 
         = \delta \gamma (n+1-t)_{\zeta^m} u^{n+1} - \gamma \delta (1+t)_{\zeta^{dt}} u^{n+1},
      \]
      and on the other hand,
      $
         (Y^t - y^m) \cdot u = (\zeta^{dt} - \zeta^m) u.
      $
      Therefore, since $m \equiv -dt \mod n$, we have
      \[
         \zeta^{-m} - \zeta^m = \gamma \delta \left( (n+1-t)_{\zeta^m} - (1+t)_{\zeta^{-m}} \right).
      \]
      Note that $\ord(\zeta^m)$ divides $n$ and that $1 < 1+t \leq n$.
      Thus, using Lemma~\ref{lem:quantum},
      \begin{align*}
         (n + 1 - t)_{\zeta^m} - (1+t)_{\zeta^{-m}} &~=~ \zeta^m (n - t)_{\zeta^m} + 1 + \zeta^m (n - 1 - t)_{\zeta^m} 
         ~=~ (\zeta^m + 1) (n-t)_{\zeta^m}. 
      \end{align*}
      Therefore, if $\zeta^m \neq -1$ (or equivalently, if $m \neq n/2$), then since $(\zeta^m + 1)(\zeta^{-m} - 1) = \zeta^{-m} - \zeta^m$, we have
      \[
         \gamma \delta = \frac {\zeta^{-m} - 1}{(n-t)_{\zeta^m}}.
      \]
      If $\zeta^m = -1$, then $\zeta^{-m} - \zeta^m = 0$,
      so we gain no new restrictions on $\delta$.
      
      We also have $YX \cdot u = \delta Y \cdot u^{1-t} = \delta \zeta^{d(1-t)} u^{1-t}$, and $\zeta^m XY \cdot u = \zeta^{m+d} X \cdot u = \delta \zeta^{m+d} u^{1-t}$.
      Therefore, $\delta = 0$ or $m+d \equiv d(1-t) \mod n$. 
      However, we already know $m \equiv -dt \mod n$, so we have no further restrictions on $\delta$ or $d$.
      
      Finally, we must have $X^N \cdot u^p = 0$ for all $p$.
      A simple calculation shows that 
      \[
         X^N \cdot u^p = \delta^N \left( \prod_{i=0}^{N-1} (p + i(n-t))_{\zeta^{dt}} \right) u^{p-Nt}.
      \]
      If $\delta = 0$, we are done.
      Otherwise, $X^N \cdot u^p = 0$ if and only if $ \ord(\zeta^{dt})$ divides some element of $\{p + i(n-t)\}_{i=0}^{N-1}$.
      Since $\ord(\zeta^{dt}) = \ord(\zeta^m) = n/m = N$ and $\text{gcd}(t, N) = 1$ by Proposition~\ref{prop:minexistence}, the set consists of $N$ distinct values mod $N$.
      Therefore, $N$ divides exactly one of them.
      Thus, $X^N \cdot u^p = 0$ for all $p$.
      
      The converse statement, that the conditions imposed on $\delta$ and $d$ are sufficient for making $A(H_n(\zeta,m,t))$ a $D(H_n(\zeta,m,t))$-module algebra, is straightforward to check.
   \end{proof}
   
Note that this result generalizes the work of Montgomery and Schneider (stated in Theorem~\ref{thm:taftext}) and shows that there are other Hopf algebras closely related to Taft algebras, for which there is a unique extension of the action of $H$ on $A(H)$ to $D(H)$, namely $H_n(\zeta,m,1)$ for any $m \mid n$ with $m \neq n/2$.
   
\begin{corollary}
   Suppose $H_n(\zeta,m,t)$-module algebra structures on $A \defeq A(H_n(\zeta, m, t))$ as in Notation~\ref{not:ah} exist.
   If $m \neq n / 2$ (e.g., if $n$ is odd), then there are precisely $t$ ways to extend this action to make $A$ a $D(H_n(\zeta, m, t))$-module algebra.
   In particular, if $t = 1$, then the desired $H_n(\zeta,m,t)$-module algebra structure on $A(H_n(\zeta, m, t))$ exists, and the way to extend the action to $D(H_n(\zeta, m, t))$ is unique.
   
   If $m = n/2$, then in order to extend the action of $H_n(\zeta,m,t)$ on $A$ to an action of $D(H_n(\zeta,m,t))$, there are $t$ ways to define the action of the generator $Y$ and the choice for the action of $X$ is parametrized by $\kk$.
\end{corollary}

   \begin{proof}
      By Proposition~\ref{prop:minexistence}, $t | m$.
      Thus, there are $t$ distinct choices for $d$ such that $0 < d < n$ and $m \equiv -dt \mod n$.
      If $m \neq n/2$, the action of $X$ is fixed by Theorem~\ref{thm:ahnzmt}.
      Otherwise, any choice of $\delta \in \kk$ will suffice to define the action of $X$.
   \end{proof}
   
While the Hopf algebra $H_n(\zeta,m,t)$ generalize the Taft algebras as bosonizations of quantum linear spaces over finite cyclic groups, there are other coradically graded generalizations and directions to consider for further study.

\begin{question} \label{q:nichols}
  What can be said about Question~\ref{questions} for quantum linear spaces of higher rank and/or over abelian non-cyclic groups?
  What about for braided vector spaces of different Cartan types (i.e. other than $A_1^\theta$) which can be realized in ${}^\Gamma_\Gamma \mc{YD}$?
  In particular, is there an even larger class of Hopf algebras which generalize the Taft algebra case in having a unique extension from the action of $H$ and $A$ to an action of $D(H)$ on $A$?
\end{question}

For instance, one could start by considering the actions of finite-dimensional pointed Hopf algebras presented in work of Etingof and Walton \cite{ew-phaof1, ew-phaof2}.

We consider Question~\ref{questions} for non-coradically graded finite-dimensional pointed Hopf algebras in the remainder of this work.

\section{The generalized Taft algebra, $T(4,2,1)$} \label{sect:gentaft}

Recall Definition~\ref{def:gentaft}:
For $n,N \in \N$ with $N \mid n$, a primitive $N^{th}$ root of unity $q$ in $\kk$, and $\alpha \in \kk$, the generalized Taft algebra $T(n,N,\alpha)$ is the Hopf algebra generated by a grouplike element $g$ and a $(g,1)$-skew primitive element $x$, subject to the relations
\[
   g^n = 1, \quad x^N = \alpha (g^N - 1), \quad gx = q xg.
\]
Note that by scaling $x$, we can assume without loss of generality that $\alpha = 0$ or $\alpha = 1$.
We saw that the algebras $H_n(\zeta, m, t)$ included the case that $\alpha = 0$, i.e. the coradically graded case.
Here, we will consider the simplest non-coradically graded case: when $n = 4$, $N = 2$, and $\alpha = 1$.
As an algebra, $T(4,2,1)$ is generated by a grouplike element $g$ and a $(g, 1)$-skew primitive element $x$, subject to the relations 
\[
   g^4 = 1, \quad 
      x^2 = g^2 - 1, \quad
      gx = -xg .
\]
As in the previous section (see Lemma~\ref{lem:primitives} and Corollary~\ref{cor:hnzmtinnerfaithful}), we will determine the primitive elements, as a means of determining if an action is inner-faithful.

\begin{lemma}
   For $b \in \{0,1,2,3\}$, we have that
   \[
      P_{g^b, 1}(T(4,2,1)) = \begin{cases}
         \kk x +\kk(g^b - 1), & \text{ if $b\equiv1 \mod 4$} \\
        \kk(g^b - 1), & \text{ otherwise}.
      \end{cases}
   \]
\end{lemma}

   \begin{proof}
      Let $\Phi = \sum_{0 \leq i < 2, 0 \leq j < 4} \alpha_{i,j} x^i g^j$ and suppose $\Phi \in P_{g^b, 1}(H)$ for some fixed $b \in \Z$.
      Then on the one hand, $\Delta(\Phi) = g^b \tens \Phi + \Phi \tens 1$. 
      On the other hand,
      \begin{align*}
         \Delta(\Phi) 
            &= \sum_{i=0}^1 \sum_{j=0}^3 \alpha_{i,j} \Delta(x)^i \Delta(g)^j
            = \sum_{i=0}^1 \sum_{j=0}^3 \alpha_{i,j} \sum_{k=0}^i x^{i-k} g^{j+k} \tens x^k g^j \\
            &= \sum_{k=0}^1 \sum_{j=0}^3 \left( \sum_{i=k}^1 \alpha_{i,j} x^{i-k} g^{j+k} \right) \tens x^k g^j.
      \end{align*}
      By comparing $\blank \tens 1$ terms, we see that 
      \[
         \Phi + \alpha_{0,0} g^b = \alpha_{0,0} 1 + \alpha_{1,0} x.
      \]
      Therefore, we have $P_{g^b, 1}(T(4,2,1)) \subseteq \kk x +\kk(1 - g^b)$.
      Now note that $x \in P_{g,1}(T(4,2,1))$ and that $1 - g^b \in P_{g^b, 1}(T(4,2,1))$ for any $b$.
   \end{proof}

\begin{corollary} \label{cor:t421innerfaithful}
   A left $T(4,2,1)$-module $M$ is inner-faithful if and only if $G(T(4,2,1)) = \langle g \rangle$ acts faithfully and $x \cdot M \neq 0$.
\end{corollary}

   \begin{proof}
      The proof is essentially the same as that of Corollary~\ref{cor:hnzmtinnerfaithful}.
   \end{proof}
   
Next, we consider the possible structures of $A(T(4,2,1))$.

\subsection{The structure of $A(T(4,2,1))$} \label{sect:t421actions}

Since the group of grouplike elements, $G(T(4,2,1))$, is cyclic of order~4, the module algebra structure $A(T(4,2,1))$ in Notation~\ref{not:ah} must be isomorphic to $\kk[u] / (u^4 - 1)$ as an algebra. 
We determine all such possible $T(4,2,1)$-module structures.

\begin{proposition} \label{prop:at421}
   Let $A = \kk[u]/(u^4 - 1)$. 
   By defining $g \cdot u = \zeta u$ for $\zeta$ a fourth root of unity and $x \cdot u = \gamma u^3$ for $\gamma \in \kk$ satisfying $\gamma^2 = 2\zeta$, we obtain that $A = A(T(4,2,1))$ is a $T(4,2,1)$-module algebra as in Notation~\ref{not:ah}.
   Moreover, this gives all the possible $T(4,2,1)$-module algebra structures on $A(T(4,2,1))$.
\end{proposition}

   \begin{proof}
      For the first statement, it is easy to check that $A$, as defined, will be a $T(4,2,1)$-module algebra. 
      By Corollary~\ref{cor:t421innerfaithful}, since $x \cdot u \neq 0$, the action on $A$ is inner-faithful.
      
      To see that these are the only possible $T(4,2,1)$-module algebra structures on $A(T(4,2,1))$ as in Notation~\ref{not:ah}, fix such a $T(4,2,1)$-module algebra structure on $A(T(4,2,1))$.
      By Remark~\ref{rem:astructure}, we have that
      \[
         A  = \bigoplus_{ i = 0 }^3 A_i \text{ where } A_i = \{ a \in A \mid g \cdot a = \zeta^i a \} = \kk u^i.
      \]
      Now, since $g \cdot x \cdot u = - x \cdot g \cdot u = - \zeta x \cdot u = \zeta^3  x \cdot u$, we have that $x \cdot u \in A_3 =\kk u^3$. 
      Therefore, $ x \cdot u = \gamma u^3 $ for some $\gamma \in \kk$. 
      We must also have that $x^2 \cdot u = (g^2 - 1) \cdot u$. 
      First, we have $ x^2 \cdot u = \gamma \; x \cdot u^3 $ and using the $ H $-module algebra structure,
      \begin{equation*}
         x \cdot u^3 = (g \cdot u )^2 ( x \cdot u ) + ( g \cdot u ) ( x \cdot u ) ( 1 \cdot u ) + ( x \cdot u ) ( 1 \cdot u)^2 
           = \zeta \gamma u.
      \end{equation*}
      Thus, $ x^2 \cdot u = \zeta \gamma^2 u $.
      On the other hand, $(g^2 - 1) \cdot u = -u - u = -2 u$. 
      Therefore, we must have $\zeta \gamma^2 = -2$, or $\gamma^2 = 2\zeta$.
   \end{proof}
   
Note that if $\zeta$ is a primitive $n^{th}$ root of unity with $\zeta^{n/N} = q$, then $T(n, N, 1)$ is a lifting of $H_n(\zeta, 1, n/N)$: 
\[
   \text{gr}(T(n,N,1)) \cong H_n(\zeta, 1, n/N).
\]
Thus, $T(4,2,1)$ is a lifting of $H_4(\zeta, 1, 2)$ where $\zeta$ is a primitive fourth root of unity.
Now, by Proposition~\ref{prop:minexistence}, there are no $H_4(\zeta, 1, 2)$-module algebra structures on $A(H_4(\zeta, 1, 2))$ as in Notation~\ref{not:ah}, let alone extensions to the double.
Hence, it is a little surprising that there are $T(4,2,1)$-module algebra structures on $A(T(4,2,1))$.

\subsection{The dual $ T(4,2,1)^* $ and the Drinfel'd double $ D(T(4,2,1) ) $ } \label{sect:t421double}

We must now compute the Drinfel'd double of $T(4,2,1)$ so that we can examine the extensions of actions of $T(4,2,1)$ on $A(T(4,2,1))$ to actions of $D(T(4,2,1))$. 
First, we compute a presentation of the dual.
We proceed in a similar fashion to Section~\ref{subsubsect:hnzmtdual}

Let $K$ denote the algebra generated by $G$ and $X$ subject to the relations
\[
   G^4 = 1, \quad X^2 = 0, \quad GX = \zeta XG.
\]
The algebra $K$ is $8$-dimensional with basis $\{X^i G^j\}_{0 \leq i \leq 1, \ 0 \leq j \leq 3}$.
With 
\begin{gather*}
   \Delta(G) = G \tens G - 2XG^3 \tens XG, \quad 
     \Delta(X) = G^2 \tens X + X \tens 1, \\
   \epsilon(G) = 1, \quad 
     \epsilon(X) = 0, \quad 
     S(G) = G^3, \quad 
     S(X) = XG^2,
\end{gather*}
$K$ has the structure of a Hopf algebra.

\begin{proposition} \label{prop:t421dual}
  With $g,x$ denoting the generators of $T(4,2,1)$,  and $G,X$ the generators $K$, the bilinear form defined by
  \begin{equation} \label{eq:t421pairing}
     \langle X^i G^j, x^k g^\ell \rangle = \delta_{i,k} \ \zeta^{j \ell},
  \end{equation}
  is a perfect duality.
  Therefore, $T(4,2,1)^* \cong K$.
\end{proposition}

In particular, we get that the dual pairing is given on generators by 
\[
  \langle G,g \rangle = \zeta, \quad
    \langle G,x \rangle = 0, \quad
    \langle X,g \rangle = 0, \quad
    \langle X,x \rangle = 1
\]

  \begin{proof}[Proof of Proposition~\ref{prop:t421dual}]
    Note that for $0 \leq i \leq 1$ and $0 \leq j \leq 3$, we have $\Delta(x^i g^j) = \sum_{k=0}^i x^{i-k} g^{j+k} \tens x^k g^j$.    
    Now, on the one hand, we have
    \begin{align*}
       \langle X^a G^b X^c G^d, x^i g^j \rangle &= \zeta^{bc} \langle X^{a+c} G^{b+d}, x^i g^j \rangle = \delta_{a+c, i} \ \zeta^{bc + j(b+d)}.
    \end{align*}
    On the other hand,
    \begin{equation*}
       \langle X^a G^b, \com {(x^i g^j)} 1 \rangle \langle X^c G^d, \com {(x^i g^j)} 2 \rangle 
          = \sum_{k=0}^i \langle X^a G^b, x^{i-k} g^{j+k} \rangle \langle X^c G^d, x^k g^j \rangle \\
          = \delta_{a + c, i} \ \zeta^{b(j+c) + dj}.
    \end{equation*}
    The proof that $\langle X^a G^b, x^i g^j x^k g^\ell \rangle = \langle \com {(X^a G^b)} 1, x^i g^j \rangle \langle \com {(X^a G^b)} 2, x^k g^\ell \rangle$ is similar.
    We also have that
    \[
       \langle 1, x^i g^j \rangle = \delta_{0,i} = \epsilon(x^i g^j) \quad \text{and} \quad \langle X^a G^b, 1 \rangle = \delta_{a,0} = \epsilon(X^a G^b).
    \]
    Finally, recalling that $\zeta$ is a primitive fourth root of unity, we obtain that
    \[
       \begin{array}{r l l l l}
         \langle S(G^b), g^j \rangle
           & = \langle G^{3b}, g^j \rangle 
           & = \zeta^{3jb}
           & = \langle G^b, g^{3j} \rangle
           & = \langle G^b, S(g^j) \rangle, \\
         \langle S(G^b), x g^j \rangle
           & = \langle G^{3b}, x g^j \rangle
           & = 0
           & = \langle G^b, (-1)^j x g^{-1-j} \rangle
           & = \langle G^b, S(x g^j) \rangle, \\
         \langle S(X G^b), g^j \rangle
           & = \langle \zeta^{3b} X G^{3b + 2}, g^j \rangle
           & = 0
           & = \langle X G^b, g^{3j} \rangle
           & = \langle X G^b, S(g^j) \rangle, \\
         \langle S(X G^b), x g^j \rangle 
           & = \langle \zeta^{3b} X G^{3b + 2}, x g^j \rangle
           & = (-1)^j \zeta^{b(-1-j)}
           & = \langle X G^b, (-1)^j x g^{-1-j} \rangle
           & = \langle X G^b, S(x g^j) \rangle.  
       \end{array}
    \]
    Therefore, the bilinear map is in fact a duality.
    We now need to establish it is perfect by showing that $\phi: K \to T(4,2,1)^*$ defined by $\phi(u)(x) = \langle u, x \rangle$ is injective.
    Let $f = \sum_{a = 0}^1 \sum_{b = 0}^3 \alpha_{a,b} X^a  G^b$ with $\alpha_{a,b} \in \kk$ and suppose $\phi(f) = 0$.
    Then for any $i,j$,
    \[
       0 = \phi(f)(x^i g^j) 
          = \langle f, x^i g^j \rangle
          = \sum_{a = 0}^1 \sum_{b = 0}^3 \alpha_{a,b} \langle X^a G^b, x^i g^j \rangle 
          = \sum_{a = 0}^1 \sum_{b = 0}^3 \alpha_{a,b} \delta_{a,i} \zeta^{bj}
          = \sum_{b=0}^3 \alpha_{i,b} \zeta^{bj}.
    \]
    Let $\beta_{i,j}$ denote $\sum_{b=0}^3 \alpha_{i,b} \zeta^{bj}$.
    For any fixed $i$ and $k$,
    \[
       0 = \sum_{j=0}^3 \zeta^{-jk} \beta_{i,j}
          = \sum_{j=0}^3 \zeta^{-jk} \sum_{b=0}^3 \alpha_{i,b} \zeta^{bj}
          = \sum_{b=0}^3 \left( \sum_{j=0}^3 \zeta^{j(b-k)} \right) \alpha_{i,b}
          = 4 \alpha_{i,k}
    \]
    Therefore, since each $\alpha_{i,k} = 0$, we get that $f = 0$, so $\phi$ is injective, and the duality is perfect.
    Thus, $K \cong T(4,2,1)^*$.
  \end{proof}

We can now prove the following result.

\begin{proposition}
   The Drinfel'd double of $D(T(4,2,1))$ of $T(4,2,1)$ is generated by $g, x, G, \text{ and } X$, subject to the relations 
   \begin{gather*}
      G^4 = g^4 = 1, \quad  x^2 = g^2 - 1, \quad X^2 = 0, \quad gx = -xg, \quad GX = \zeta XG \\
      gG = Gg, \quad gX = -Xg, \quad xX - Xx = G^2 - g, \quad xG - \zeta Gx = 2XG(\zeta g - G^2).
   \end{gather*}
   The coalgebra structure is determined by
   \begin{gather*}
      \Delta(g) = g \tens g, \quad 
         \Delta(x) = g \tens x + x \tens 1, \quad
         \Delta(G) = G \tens G - 2XG \tens X G^3, \quad
         \Delta(X) = 1 \tens X + X \tens G^2, \\
      \epsilon(g) = \epsilon(G) = 1, \quad
         \epsilon(x) = \epsilon(X) = 0.
   \end{gather*}
\end{proposition}

   \begin{proof}
      The generators and top row of relations follows from Lemma~\ref{lem:gens} and Proposition~\ref{prop:t421dual}.
      For the rest, first note that in $K$ and $T(4,2,1)$, we have
      \begin{gather*}
         \Delta^2(G) = G \tens G \tens G -2 G \tens XG^3 \tens XG - 2 XG^3 \tens XG \tens G - 2 XG^3 \tens G^3 \tens XG, \\
         \Delta^2(X) = G^2 \tens G^2 \tens X + G^2 \tens X \tens \epsilon + X \tens \epsilon \tens \epsilon , \\
         \Delta^2(g) = g \tens g \tens g , \quad
         \Delta^2(x) = g \tens g \tens x + g \tens x \tens 1 + x \tens 1 \tens 1 , \\
         S^{-1}(g) = g^{-1} = g^3, \quad
         S^{-1}(x) = - x g^3 = g^3 x.
      \end{gather*}
      Thus, using \eqref{eq:doublemixed} and \eqref{eq:t421pairing}, we have the following computations:
      \begin{align*}
         g G &= \langle G, g^3 \rangle \langle G, g \rangle G g %
            - 2 \langle G, g^3 \rangle \langle XG, g \rangle XG^3 g %
            - 2 \langle XG^3, g^3 \rangle \langle G, g \rangle XG g %
            - 2 \langle XG^3, g^3 \rangle \langle XG, g \rangle G^3 g   
         = Gg,
       \end{align*}
       \begin{align*}
         g X &= \langle G^2, g^3 \rangle \langle X, g \rangle G^2 g 
            + \langle G^2, g^3 \rangle \langle \epsilon, g \rangle X g 
            + \langle X, g^3 \rangle \langle \epsilon, g \rangle g  
         = \zeta^2 Xg 
         = - Xg,
      \end{align*}
      \begin{alignat*}{3}
         x X &= \langle G^2, g^3 x \rangle \langle X, g \rangle G^2 g %
            &&+ \langle G^2, g^3 x \rangle \langle \epsilon, g \rangle X g %
            &&+ \langle X, g^3x \rangle \langle \epsilon, g \rangle g    \\
         &\hspace{3mm} + \langle G^2, 1 \rangle \langle X, g \rangle G^2 x %
            &&+ \langle G^2, 1 \rangle \langle \epsilon, g \rangle X x %
            &&+ \langle X, 1 \rangle \langle \epsilon, g \rangle \epsilon x    \\
         &\hspace{3mm} + \langle G^2, 1 \rangle \langle X, x \rangle G^2 %
            &&+ \langle G^2, 1 \rangle \langle \epsilon, x \rangle X %
            &&+ \langle X, 1 \rangle \langle \epsilon, x \rangle 1     \\
         &= -g + Xx + G^2,
       \end{alignat*}
       \begin{alignat*}{4}
         x G &= \langle G, g^3 x \rangle \langle G, g \rangle G g %
            &&-2 \langle G, g^3 x \rangle \langle XG, g \rangle XG^3 g %
            &&-2 \langle XG^3, g^3x \rangle \langle G, g \rangle XG g %
            &&-2 \langle XG^3, g^3 x \rangle \langle XG, g \rangle G^3 g   \\
         &\hspace{3mm} + \langle G, 1 \rangle \langle G, g \rangle G x %
            &&-2 \langle G, 1 \rangle \langle XG, g \rangle XG^3 x %
            &&-2 \langle XG^3, 1 \rangle \langle G, g \rangle XG x %
            &&-2 \langle XG^3, 1 \rangle \langle XG, g \rangle G^3 x   \\
         &\hspace{3mm} + \langle G, 1 \rangle \langle G, x \rangle G^3 %
            &&-2 \langle G, 1 \rangle \langle XG, x \rangle XG^3 %
            &&-2 \langle XG^3, 1 \rangle \langle G, x \rangle XG %
            &&-2 \langle XG^3, 1 \rangle \langle XG, x \rangle G^3  \\
         &= -2 XGg + \zeta Gx - 2&&XG^3.
      \end{alignat*}
      
      \vspace{-.2in}
      
   \end{proof}

\subsection{(The lack of) extensions to $ D(T(4,2,1))$}

We now come to the surprising result that the $ T(4,2,1) $-module algebra structures computed in Section~\ref{sect:t421actions} are not $ D(T(4,2,1)) $-module algebras.

\begin{proposition} \label{prop:t421extensions}
   The action of $ T(4,2,1) $ on $ A(T(4,2,1)) $ cannot extend to an action of $ D(T(4,2,1)) $ on $ A $ in any way to make $ A $ a $ D(T(4,2,1))$-module algebra.
\end{proposition}

   \begin{proof}
      Suppose by contradiction that we have such an extension. 
      By Proposition~\ref{prop:at421}, we have that \linebreak ${ A = \kk[u] / (u^4 - 1) }$, with $ g \cdot u = \zeta u $, and $ x \cdot u = \gamma u^3 $, where $\gamma^2 = 2\zeta $. 
      By the relation $ gX = - Xg $, we have that
      \[
         g \cdot X \cdot u = - X \cdot g \cdot u = - \zeta X \cdot u = \zeta^3 X \cdot u.
      \]
      Therefore, $ X \cdot u \in A_3 = \kk u^3 $, so $ X \cdot u = \delta u^3 $ for some $ \delta \in \kk $.
      Similarly, by the relation $ gG = Gg $ of $D(T(4,2,1))$, we have $g \cdot G \cdot u = G \cdot g \cdot u = \zeta G \cdot u$, so $ G \cdot u \in A_1 =\kk u $. 
      Therefore, $ G \cdot u = \eta u $ for some $ \eta \in \kk $. 
      Since $ G^4 = 1 $, $ \eta = \zeta^i $ for some integer $ i $. 
      
      In $ D(T(4,2,1)) $, we have 
      \[
         \Delta^2 ( X ) = X \tens G^2 \tens G^2 + 1 \tens X \tens G^2 + 1 \tens 1 \tens X.
      \]
      Thus, 
      \begin{align*}
         X \cdot u^3 &= (X \cdot u)( G^2 \cdot u )^2  %
            + u ( X \cdot u ) (G^2 \cdot u )  %
            + u^2 (X \cdot u ) \\
         &= (\delta u^3)(\eta^4 u^2) %
            + u (\delta u^3)(\eta^2 u) %
            + u^2 (\delta u^3) 
          = (\eta^2 + 2) \delta u.
      \end{align*}
      Using this calculation we have $0 = X^2 \cdot u = \delta X \cdot u^3 = (\eta^2 + 2) \delta^2 u $.
      Since $\eta$ is a fourth root of unity, $\eta^2 \neq -2$, so we must have $\delta = 0$.
      Therefore, $X$ acts by zero.
      
      Hence, on one hand, $(xX - Xx) \cdot u = 0$.
      On the other hand, since $xX - Xx = G^2 - g$, we have 
      \[
         (xX - Xx) \cdot u = G^2 \cdot u - g \cdot u = (\eta^2 - \zeta) u.
      \]
      Thus, since $\eta$ is a power of the fourth root of unity $\zeta$, we arrive at a contradiction: $\eta^2 = \zeta$.
   \end{proof}

\section{The Frobenius-Lusztig kernel, $u_q(\mf{sl}_2)$} \label{sect:uqsl2}

The next algebra we study is the Frobenius-Lusztig kernel, $u_q(\mf{sl}_2)$.
It is well-known that $u_q(\mf{sl}_2)$ contains two isomorphic copies of Taft algebras, which generate the whole algebra.
Let $q$ be a primitive $n^{th}$ root of unity, with $n$ odd.
Recall that $u_q(\mf{sl}_2)$ is generated by grouplike $K$, a $(1,K)$-skew primitive $E$, and a $(K^{-1},1)$-skew primitive $F$, subject to the relations:
\[
   K^n = 1, \quad E^n = F^n = 0, \quad KE = q^2 EK, \quad KF = q^{-2} FK, \quad [E,F] = \frac{K - K^{-1}}{q - q^{-1}}.
\]
In a sense, the Taft algebra $T_n(q)$ is like a Borel subalgebra of $u_q(\mf{sl}_2)$.
More precisely, with the decomposition, $T_n(q) \cong \mf B(V) \#\kk \Gamma $ as at the beginning of Section~\ref{sect:taft}, $\mf B(V) \cong u_q^+(\mf{sl}_2)$ (\cite[Theorem~4.3]{as-pha}).
For more on $u_q(\mf{sl}_2)$, and the computation of its dual and Drinfel'd double, see Appendix \ref{sect:uqdouble}.

To help us determine when an action of $u_q(\mf{sl}_2)$ is inner-faithful, we have the following:

\begin{proposition} \label{prop:uqprims}
   Let $0 \leq b < n$. Then
   \[
      P_{K^b, 1}(u_q(\mf{sl}_2)) = \begin{cases}
         \kk (K^{-1} - 1) + \kk F + \kk EK^{-1}, & \text{ if }b \equiv -1 \mod n \\
         \kk (K^b  -1), & \text{ otherwise}.
      \end{cases}
   \]
\end{proposition}

   \begin{proof}
      For convenience, we let $\hat E = EK^{-1}$, and note that $u_q(\mf{sl}_2)$ is generated by $\hat E, F, K$ and that \linebreak $\{K^\ell \hat E^i F^j\}_{0 \leq \ell,i,j \leq n-1}$ is a basis.
      Using $q$-binomial coefficients and \eqref{eq:skewbinom}, one sees that 
      \begin{equation} \label{eq:uqdeltaonbasis}
         \Delta(K^\ell \hat E^i F^j) = \sum_{s=0}^i \sum_{t=0}^j {\binom i s}_{q^2} {\binom j t}_{q^{-2}} q^{2t(i-s)} K^{\ell - s - t} \hat E^{i-s} F^{j-t} \tens K^\ell \hat E^s F^t.
      \end{equation}
      Now, fix an element of $P_{K^b, 1}$, 
      \[
        \Phi = \sum_{\ell, i, j = 0}^{n-1} \alpha_{\ell,i,j} K^\ell \hat E^i F^j.
      \]
      On one hand, applying \eqref{eq:uqdeltaonbasis}, we have 
      \begin{align*}
         \Delta(\Phi) 
            &= \sum_{\ell,i,j=0}^{n-1} \sum_{s=0}^i \sum_{t=0}^j \alpha_{\ell, i, j} {\binom i s}_{q^2} {\binom j t}_{q^{-2}} q^{2t(i-s)} K^{\ell - s - t} \hat E^{i-s} F^{j-t} \tens K^\ell \hat E^s F^t \\
            &= \sum_{\ell,s,t = 0}^{n-1} \left( \sum_{i=s}^{n-1} \sum_{j=t}^{n-1} \alpha_{\ell, i, j}{\binom i s}_{q^2} {\binom j t}_{q^{-2}} q^{2t(i-s)} K^{\ell - s - t} \hat E^{i-s} F^{j-t} \right) \tens K^\ell \hat E^s F^t.
      \end{align*}
      It is worth mentioning that since $q^2$ is a primitive $n^{th}$ root of unity, none of the binomial coefficients here will vanish.
      On the other hand, since $\Phi \in P_{K^b,1}(u_q(\mf{sl}_2))$, we have $\Delta (\Phi) = K^b \tens \Phi + \Phi \tens 1$.
      By comparing the coefficients of the $\blank \tens 1$ terms, we must have that
      \begin{equation} \label{eq:ysimp1}
         \sum_{i=0}^{n-1} \sum_{j=0}^{n-1} \alpha_{0, i, j} \hat E^{i} F^{j} = \Phi + \alpha_{0,0,0} K^b.
      \end{equation}
      Therefore, isolating $\Phi$ in \eqref{eq:ysimp1} and applying $\Delta$ with the use of \eqref{eq:uqdeltaonbasis}, we get
      \begin{equation} \label{eq:deltay2}
         \Delta(\Phi) = \sum_{s,t = 0}^{n-1} \left( \sum_{i=s}^{n-1} \sum_{j=t}^{n-1} \alpha_{0, i, j}{\binom i s}_{q^2} {\binom j t}_{q^{-2}} q^{2t(i-s)} K^{- s - t} \hat E^{i-s} F^{j-t} \right) \tens \hat E^s F^t - \alpha_{0,0,0} K^b \tens K^b.
      \end{equation}
      By comparing the coefficients of the $\blank \tens \hat E$ terms of \eqref{eq:deltay2} and $\Delta(\Phi) = K^b \tens \Phi + \Phi \tens 1$, we must have
      \begin{equation} \label{eq:Eterms}
         \sum_{i=1}^{n-1} \sum_{j=0}^{n-1} \alpha_{0, i, j}{\binom i 1}_{q^2} K^{- 1} \hat E^{i-1} F^{j} = \alpha_{0,1,0} K^b.
      \end{equation}
      Thus, if $i>1$ and $j \geq 0$ or if $i \geq 1$ and $j > 0$, $\alpha_{0,i,j} = 0$.
      Similarly, by comparing the $\blank \tens F$ terms, we must have
      \begin{equation} \label{eq:Fterms}
         \sum_{i=0}^{n-1} \sum_{j=1}^{n-1} \alpha_{0, i, j}{\binom j 1}_{q^{-2}} q^{2i} K^{- 1} \hat E^{i} F^{j-1} = \alpha_{0,0,1} K^b.
      \end{equation}
      Thus, if $j>1$ and $i \geq 0$ or if $j \geq 1$ and $i > 0$, $\alpha_{0,i,j} = 0$.
      Therefore, we have 
      \[
         \Phi = \alpha_{0,0,0}(1 - K^b) + \alpha_{0,0,1} F + \alpha_{0,1,0} \hat E.
      \]
      Moreover, one sees from \eqref{eq:Eterms} and \eqref{eq:Fterms} that if $b \not \equiv -1 \mod n$, then all $\alpha_{0,i,j} = 0$, except when $i = j = 0$.
      We already know that $F$ is $(K^{-1},1)$-skew primitive, and it is not hard to see that $\hat E$ is as well.
   \end{proof}
   
The following is a direct result of Corollary~\ref{cor:ifprims} and Proposition~\ref{prop:uqprims}.
   
\begin{corollary} \label{cor:uqinnerfaithful}
   A $u_q(\mf{sl}_2)$-module algebra is inner-faithful if and only if $G(H)$ acts faithfully, and if no nonzero element of $\kk(1 - K^{-1}) + \kk F + \kk EK^{-1}$ acts by zero.
   \qed
\end{corollary}

\subsection{The structure of $A(u_q(\mf{sl}_2))$ and extensions to $D(u_q(\mf{sl}_2))$}

We now consider $u_q(\mf{sl}_2)$-module algebra structures on $A(u_q(\mf{sl}_2))$ as in Notation~\ref{not:ah}.
By definition, $A = \kk[u]/(u^n - 1)$.
To see the possible module structures of $A$, we use the following result of Montgomery and Schneider.
The original statement was for $q$ a primitive $2n^{th}$ root of unity.
However, their proof is also valid for the case we are interested in, since it only relies on the fact that $q^2$ is a primitive $n^{th}$ root of unity so that $ H_1 =  \kk \langle K^{-1}, F \rangle \cong T_{n} ( q^{-2} ) $ and $ H_2 = \kk \langle K^{-1}, E K^{-1} \rangle \cong T_{n} ( q^2 ) $.

\begin{proposition}[{\cite[Corollary~3.2]{montschneid}}] \label{prop:montuq}
   Let $ A $ be an $ n $-dimensional $\kk$-algebra with no non-zero nilpotent elements, and assume that $ A $ is a $ u_q(\mf{sl}_2) $-module algebra such that $ F \cdot A \neq 0 $ (or that $ E \cdot A \neq 0 $). 
   Then there exists $ u \in A $ and $\beta, \gamma, \delta \in \kk $, all nonzero, such that 
   \begin{enum}
      
      \item
         $ A = \kk(u), \ u^n = \beta, \text{ and } K \cdot u = q^2 u $;
         
      \item
         $ F \cdot u = \gamma 1 \text{ and } E \cdot u = \delta u^2 $;
         
      \item
         $ \gamma \delta = -q$.

   \end{enum}
   Moreover $ u $ is unique up to a scalar multiple.  
   \qed
\end{proposition}

We point out here that by Corollary~\ref{cor:uqinnerfaithful}, the assumption that $ F \cdot A \neq 0 $ or $E \cdot A \neq 0$ is necessary for the action to be inner-faithful, and that the actions on $A$ described are in fact inner-faithful, because no nonzero element of $\kk (1-K^{-1}) + \kk F + \kk EK^{-1}$ acts by zero.
Therefore, by scaling $u$, Proposition~\ref{prop:montuq} classifies the $u_q(\mf{sl}_2)$-module algebra structures on $A(u_q(\mf{sl}_2))$ as in Notation~\ref{not:ah}.
It turns out that the action of $u_q(\mf{sl}_2)$ on $A$ extends to an action of $D(u_q(\mf{sl}_2))$ in two distinct ways. 

In Appendix~\ref{sect:uqdouble}, we compute an algebra presentation of $D(u_q(\mf{sl}_2))$.
(See Theorem~\ref{thm:uqdouble} for the presentation.)
Therefore, we have all the tools we need to classify extensions to the double.

\begin{theorem} \label{thm:uqextension}
   Fix a $u_q(\mf{sl}_2)$-module algebra structure on $A(u_q(\mf{sl}_2)) = \kk[u]/(u^n - 1)$ as in Notation~\ref{not:ah} by 
   \[
      K \cdot u = q^2 u, \quad F \cdot u = \gamma 1, \quad E \cdot u = \delta u^2,
   \]
   with $q$ a primitive $n^{th}$ root of unity, and $\gamma \delta = -q$.
   Recall the presentation of $D(u_q(\mf{sl}_2))$ as in Theorem~\ref{thm:uqdouble}.
   If the action of $u_q(\mf{sl}_2)$ on $A$ extends to an action of $D(u_q(\mf{sl}_2))$ so that $A$ is a $D(u_q(\mf{sl}_2))$-module algebra, then the action is specified by one of the following two conditions:
   \[
     \begin{array}{r l l l l}
       \text{(i)}
         & a \cdot u = q u,
         & b \cdot u = \gamma (q - q^{-1}) 1, 
         & c \cdot u  = 0, 
         & d \cdot u = q^{-1} u, \quad \text{or} \\
       \text{(ii)}
         & a \cdot u = q^{-1} u, 
         & b \cdot u = 0, 
         & c \cdot u = \gamma^{-1} (q - q^{-1}) u^2, 
         & d \cdot u = q u.
     \end{array}
   \]
   Conversely, by defining the action of $a$, $b$, $c$, and $d$ by either (i) or (ii), an action of $u_q(\mf{sl}_2)$ on $A$ extends to an action of $D(u_q(\mf{sl}_2))$.
\end{theorem}

   \begin{proof}
      Since $K \cdot u = q^2 u$, we use notation similar to that in Remark~\ref{rem:astructure}:
      \[
         A_i = \{a \in A \mid K \cdot a = q^{2i} a \} = \kk u^i.
      \]
      First, since $Ka = aK$, we have $K \cdot a \cdot u = a \cdot K \cdot u = q^2 a \cdot u$, so $a \cdot u \in A_1 =\kk u$. 
      Similarly, since $Kb = q^{-2} bK$, $Kc = q^2 cK$, and $Kd = dK$, we get that $ b \cdot u \in A_0 $, $ c \cdot u \in A_2 $, and $ d \cdot u \in A_1 $. 
      Therefore, there exists $\theta_a, \theta_b, \theta_c, \theta_d \in \kk$ such that
      \begin{gather*}
         a \cdot u = \theta_a u, \hspace{10mm} 
         b \cdot u = \theta_b 1, \hspace{10mm} 
         c \cdot u = \theta_c u^2, \hspace{10mm} \text{and} \hspace{10mm} 
         d \cdot u = \theta_d u .
      \end{gather*}
      Now, note that $ c \cdot 1 = \epsilon ( c ) = 0 $. 
      Thus, since $ bc = cb $ and $ ad = q^{-1} bc + 1 $, we compute that 
      \[
         \theta_a \theta_d u = (ad) \cdot u = q^{-1} c \cdot (b \cdot u) + 1 \cdot u = q^{-1} \theta_b c \cdot 1 + u = u .
      \]
      Therefore, $ \theta_d = \theta_a^{-1} $. 
      Using the fact that $a^n = 1$, for some integer $ i $, we have $\theta_a = q^i$ and $\theta_d = q^{-i}$.
      Note that $b \cdot u^2 = ( b \cdot u )( a \cdot u ) + ( d \cdot u ) ( b \cdot u ) =  \theta_b \theta_a u + \theta_d \theta_b u = \theta_b ( \theta_a + \theta_d) u $. 
      Thus, 
      \[
         \theta_c \theta_b ( \theta_a + \theta_d ) u = ( bc ) \cdot u = ( cb ) \cdot u = \theta_b c \cdot 1 =  0.
      \]
      Since $\theta_a = q^i$ is an odd root of unity, $ \theta_a \neq - \theta_d (= - \theta_a^{-1}).$
      Thus, we must have 
      \begin{gather} \label{eq:phitheta}
         \theta_b = 0 \quad \text{or} \quad \theta_c = 0.
      \end{gather}
      We also compute, using $a \cdot 1 = \epsilon(a) = 1$ and $d \cdot 1 = \epsilon(d) = 1$, that 
      \begin{gather*}
         \theta_a \gamma 1 = (Fa) \cdot u = q^{-1} (aF) \cdot u + b \cdot u = (q^{-1} \gamma + \theta_b) 1 \hspace{10mm} \text{and}  \\
         \theta_d \gamma 1 = (Fd) \cdot u = q (dF) \cdot u - q^2 (bK^{-1}) \cdot u = (q \gamma - \theta_b) 1,
      \end{gather*}
      which shows that 
      \begin{gather} \label{eq:etazeta}
         \theta_a = q^{-1} + \theta_b \gamma^{-1} \quad 
         \text{and} \quad \theta_d = q - \theta_b \gamma^{-1}.
      \end{gather}
      Therefore,
      \[
         1 = \theta_a \theta_d = ( q^{-1} + \theta_b \gamma^{-1}) ( q - \theta_b \gamma^{-1} ) = 1 + ( q - q^{-1}) \theta_b \gamma^{-1} - \theta_b^2 \gamma^{-2},
      \]
      implying that $ 0 = \theta_b \gamma^{-1} ( q - q^{-1} - \theta_b \gamma^{-1}) $. 
      Since $\gamma \neq 0$, we have
      \begin{gather*}
         \theta_b = 0 \quad \text{or} \quad
         \theta_b = \gamma (q - q^{-1}).
      \end{gather*}
      The former will correspond to \textit{(ii)} and the latter to \textit{(i)}.
      In case \textit{(i)}, by \eqref{eq:phitheta}, $ \theta_c = 0 $, and by \eqref{eq:etazeta}, $ \theta_a = q $ and $ \theta_d = q^{-1} $. 
      On the other hand, in case \textit{(ii)}, by \eqref{eq:etazeta}, $ \theta_a = q^{-1} $ and $ \theta_d = q $. 
      Also, using the fact that $\gamma \delta = -q$, $ Ea = q^{-1} aE - q^{-1} c $, and $ a \cdot u^2 = ( a \cdot u )^2 + ( c \cdot u ) ( b \cdot u ) = q^{-2} u^2 $, we have 
      \[
         -\gamma^{-1} u^2 = q^{-1} \delta u^2 = (Ea) \cdot u = q^{-1} ( aE ) \cdot u - q^{-1} c \cdot u = q^{-1} \delta a \cdot u^2 - q^{-1} \theta_c u^2 
         = - ( q^{-2} \gamma^{-1} + q^{-1} \theta_c ) u^2.
      \]
      Therefore, $ \gamma^{-1} = q^{-2} \gamma^{-1} + q^{-1} \theta_c $, which implies $ \theta_c = \gamma^{-1} (q - q^{-1}) $.
      Therefore, we have shown that an action of $D(u_q(\mf{sl}_2))$ is specified by either \textit{(i)} or \textit{(ii)}.
      
      It is straightforward to check the converse: that $A $ is a $D(u_q(\mf{sl}_2))$-module algebra with either of these structures.
   \end{proof}

Perhaps unsurprisingly, while $T_n(q)$ had a unique extension of its action on $A(T_n(q))$ to its double, $u_q(\mf{sl}_2)$ has exactly two extensions of its action on $A(u_q(\mf{sl}_2))$ to its double.
We are led to ask the following.

\begin{question} \label{q:borel}
   For a semisimple finite-dimensional Lie algebra $\mf g$, is the answer to Question~\ref{questions}(c) for $u_q(\mf g)$ twice what the answer would be for a Borel subalgebra?
\end{question}

\section{Acknowledgements}

I would like to thank my advisor, Chelsea Walton, for her continual guidance and for introducing me to the world of quantum symmetry.
This work is partially supported by C. Walton's NSF grant DMS-1663775 and her Alfred P. Sloan research fellowship.


\appendix

\section{The Drinfel'd double $D( u_q(\mf{sl}_2) )$} \label{sect:uqdouble}

For Section~\ref{sect:uqsl2}, we require a presentation of the double $D(u_q(\mf{sl}_2))$.
Using the fact that $u_q(\mf{sl}_2)$ is factorizable, \cite[Theorem~2.9]{factproblems} provides a nice algebra presentation of $D(u_q(\mf{sl}_2))$ as $u_q(\mf{sl}_2) \tens u_q(\mf{sl}_2)$.
However, with this presentation, the coproduct becomes much more complicated.
The method we use to extend actions of a Hopf algebra to its double requires an uncomplicated coproduct, so we provide here a different presentation for $D(u_q(\mf{sl}_2))$.

Let $n \geq 3$ be an odd integer and let $q \in \kk$ be a primitive $n^{th}$ root of unity.
The quantum group $U_q(\mf{sl}_2)$, often called the \emph{quantized universal enveloping algebra} of $\mf{sl}_2$, is the Hopf algebra generated by grouplike elements $K$ and $K^{-1}$, a $(1,K)$-skew primitive element $E$, and a $(K^{-1},1)$-skew primitive element $F$, subject to the relations
\[
  K K^{-1} = K^{-1} K = 1, \quad
  KE = q^2 EK, \quad
  KF = q^{-2} FK, \quad
  EF - FE = \dfrac { K - K^{-1} } { q - q^{-1} }.
\]
The Frobenius-Lusztig kernel $u_q(\mf{sl}_2)$ is then the quotient of $U_q(\mf{sl}_2)$ by the (Hopf) ideal generated by $K^n - 1$, $E^n$, and $F^n$.
Note that $\{E^i F^j K^\ell\}_{0 \leq i,j,\ell < n}$ is a basis of $u_q(\mf{sl}_2)$.
We compute here a presentation of the Drinfel'd double $D(u_q(\mf{sl}_2))$.
This is accomplished by first showing that $u_q(\mf{sl}_2)$ is dual to a quotient of the \emph{quantized coordinate ring} $\mc O_q(SL_2)$.
This result is well-known (see \cite[III.7.10]{browngoodearl}), but we include here an explicit proof for completion.

The quantum group $O_q(SL_2)$ is the Hopf algebra generated by $a,b,c,d$ subject to the relations
\begin{gather*}
  ba = qab, \quad 
  ca = qac, \quad 
  db = qbd, \quad 
  dc = qcd, \quad
  bc = cb, \quad 
  ad = q^{-1} bc + 1, \quad
  da = qbc + 1,
\end{gather*}
with coalgebra structure and antipode given by
\begin{gather*}
  \Delta(a) = a \tens a + b \tens c, \quad
  \Delta(b) = a \tens b + b \tens d, \quad
  \Delta(c) = c \tens a + d \tens c, \quad
  \Delta(d) = c \tens b + d \tens d   \\
  \epsilon(a) = \epsilon(d) = 1, \quad
  \epsilon(b) = \epsilon(c) = 0, \quad
  S(a) = d, \quad
  S(b) = -qb, \quad
  S(c) = -q^{-1}c, \quad
  S(d) = a.
\end{gather*}
One can easily verify that the ideal $J$ generated by $a^n - 1$, $b^n$, $c^n$, and $d^n - 1$ is a Hopf ideal, so we define $\overline{ \mc O_q(SL_2)} \defeq \mc O_q(SL_2) / J. $ 
In $ \overline{\mc O_q(SL_2)}$, the generators $a$ and $d$ are invertible.
Using this, the relation $da = qbc + 1$ becomes vacuous.
Also, we can use the relation $ad = q^{-1}bc + 1$ to eliminate the generator $a$ from the algebra presentation of $\overline{\mc O_q (SL_2)}$. 
If we do so, all other relations involving $a$ become vacuous, so we have 
\[
   \overline {  \mc O_q ( SL_2 ) } \cong \kk \langle b, c, d \mid 
   b^n, \ 
   c^n, \ 
   d^n - 1 , \ 
   bc - cb, \ 
   db - qbd, \ 
   dc - qcd 
   \rangle
\]
as algebras.
Thus, the finite set $ \{ b^i c^j d^\ell \} _ { 0 \leq i, j, \ell \leq n-1 } $ is a basis for $ \overline {  \mc O_q ( SL_2 ) } $, and $ {\dim_\kk(\overline {  \mc O_q ( SL_2 )  } ) = n^3} $. 

The first step toward showing that $\overline{\mc O_q(SL_2)} \cong u_q(\mf{sl}_2)^*$ is exhibiting a duality between $\mc O_q(SL_2)$ and $U_q(\mf{sl}_2)$. 
This is done in \cite[VII.4]{kassel} and we recall the duality here.
Let $ V_{1,1} $ denote the highest weight $ U_q ( \mf{sl}_2 ) $-module with basis $ v_0, v_1 $ determined by 
\[
  E \cdot v_1 = v_0 ,\quad 
  F \cdot v_0 = v_1 ,\quad 
  K \cdot v_0 = q v_0 , \quad 
  K \cdot v_1 = q^{-1} v_1 ,\quad  
  E \cdot v_0 = F \cdot v_1 = 0 .
\]
In other words, if $ \rho: U_q ( \mf{sl}_2 ) \lra \End_\kk ( V_{1,1} ) $ denotes the representation, then, identifying $ \End_\kk ( V_{1,1}) $ with $ \text{M}_2 (\kk) $ on the ordered basis $ \{ v_0, v_1 \} $, we have 
$ \rho(E) = 
    \begin{pmatrix}
      0 & 1 \\ 0 & 0
    \end{pmatrix}$,
$ \rho(F) = 
    \begin{pmatrix}
      0 & 0 \\ 1 & 0
    \end{pmatrix}$,
and
$ \rho(K) =
    \begin{pmatrix}
      q & 0 \\ 0 & q^{-1}
    \end{pmatrix}$.
Now, for any element $ u \in U_q ( \mf { sl } _2 ) $, define 
\begin{equation*}
   \rho ( u ) = 
   \left( \begin{array}{cc}
   A(u) & B(u)  \\
   C(u) & D(u) \end{array} \right)
\end{equation*}
to get four elements $A$, $B$, $C$, and $D$ of $ U_q ( \mf { sl} _2 ) ^* $. 

\begin{theorem}[{\cite[VII.4.4]{kassel}}] \label{thm:uqduality} 
   Let $ \phi: \mc O_q ( SL_2 ) \lra U_q ( \mf { sl }_2 )^* $ be defined by $ \phi ( a ) = A $, $ \phi ( b ) = B $, $ \phi ( c ) = C $, $ \phi ( d ) = D $. 
   Then $ \phi $ is a Hopf algebra map, and the bilinear form $ \langle u,x \rangle = \phi ( u ) ( x ) $ realizes a duality between the Hopf algebras $ \mc O_q ( SL_2 ) $ and $ U_q ( \mf { sl}_2 ) $.  \qed 
\end{theorem}

\begin{lemma}
  For the map $\phi: \mc O_q(SL_2) \lra U_q(\mf{sl}_2)^*$ given in Theorem~\ref{thm:uqduality}, we have that $\textup{Im}(\phi) \subseteq u_q(\mf{sl}_2)^*$.
\end{lemma}

  \begin{proof}
    We only need to show that $A$, $B$, $C$, and $D$ all vanish on the (Hopf) ideal $I$ of $U_q(\mf{sl}_2)$ generated by $K^n - 1$, $E^n$, and $F^n$, which amounts to showing that $\rho(E^n) = \rho(F^n) = \rho(K^n - 1) = 0$.
    We have that $\rho(E^n) = \rho(F^n) = 0$ because $\rho(E)$ and $\rho(F)$ each have nilpotency order 2, while $n \geq 3$. 
    That $\rho(K^n - 1)$ = 0 follows because $q$ is an $n^{th}$ root of unity.
  \end{proof}
  
We now have a Hopf algebra map $\phi: \mc O_q(SL_2) \lra u_q(\mf{sl}_2)^*$. 
We wish to show that $\phi$ induces an isomorphism of Hopf algebras $\overline \phi: \overline{\mc O_q(SL_2)} \lra u_q(\mf{sl}_2)^*$.
To do this, we will need the following calculations, which can be verified using the pairing from Theorem~\ref{thm:uqduality}.

\begin{lemma} \label{lem:eifj}
  For $ i, j $ nonnegative integers, 
  \begin{gather*}
    \langle a^n, E^i \rangle = \langle d^n, E^i \rangle=  \delta_{i,0}, \quad
    \langle a^n, F^j \rangle = \langle d^n, F^j \rangle = \delta_{j,0}, \quad
    \langle b^n, E^i \rangle =\langle  b^n, F^j \rangle = 0 .
  \end{gather*}
  
  \vspace{-.25in} \qed
\end{lemma}

\begin{proposition}
  The map $ \phi $ induces a Hopf algebra map $ \overline \phi :\overline{\mc O_q(SL_2)} \lra u_q(\mf{sl}_2)^*$ determined by $ \phi = \overline \phi \circ \pi $, where $ \pi : \mc O_q ( SL_2 ) \lra \overline{ \mc O_q ( SL _ 2 ) }$ is the usual projection.
  Hence, the bilinear form $\langle u, x \rangle = \overline \phi(u)(x)$ realizes a duality between the Hopf algebras $\overline{\mc O_q(SL_2)}$ and $u_q(\mf{sl}_2)$.
\end{proposition}

  \begin{proof}
    We need to show that $ \phi $ vanishes on $a^n - 1$, $b^n$, $c^n$, and $d^n - 1$.  
    Note that $(b \tens c)(a \tens a) = q^2 (a \tens a)(b \tens c)$, and that $q^2$ is a primitive $n^{th}$ root of unity because $n$ is odd.
    Thus, by \cite[Corollary~7.2.2]{radford}, $ \Delta ( a^n ) = ( a \tens a + b \tens c ) ^n = a^n \tens a^n + b^n \tens c^n $, and similarly for $ \Delta ( b^n ), \Delta ( c^n ) $, and $ \Delta ( d^n ) $. 
    Thus, using Lemma~\ref{lem:eifj} and the duality of Theorem~\ref{thm:uqduality}, we compute for $ i, j, $ and $ k $ nonnegative integers,
     \begin{align*}
       \langle a^n,  E^i F^j K^\ell \rangle
         &= \langle a^n,  E^i F^j \rangle \langle a^n,  K^\ell \rangle + \langle b^n,  E^i F^j\rangle \langle c^n, K^\ell \rangle    
         ~=~ q^{\ell n } \langle a^n,  E^i F^j \rangle + \langle b^n,  E^i F^j \rangle \langle c, K^\ell \rangle ^n \\
         &= \langle a^n,  E^i F^j \rangle 
         ~=~ \langle a^n,  E^i \rangle \langle a^n,  F^j \rangle + \langle b^n,  E^i \rangle \langle c^n, F^j \rangle 
         ~=~ \delta_{i,0} \delta_{j,0}, \\
       \langle b^n, E^i F^j K^\ell \rangle 
         &= \langle a^n,  E^i F^j \rangle \langle b^n,  K^\ell \rangle + \langle b^n,  E^i F^j \rangle \langle d^n, K^\ell \rangle 
         ~=~ \langle b^n,  E^i F^j \rangle
         ~=~ 0, \\
       \langle c^n, E^i F^j K^\ell \rangle 
         &= \langle c^n, E^i F^j \rangle \langle a^n,  K^\ell \rangle + \langle d^n, E^i F^j \rangle \langle c^n, K^\ell \rangle 
         ~=~ \langle c^n, E^i F^j \rangle 
         ~=~ 0, \\
       \langle d^n, E^i F^j K^\ell \rangle 
         &= \langle c^n, E^i F^j \rangle \langle b^n, K^\ell \rangle + \langle d^n, E^i F^j \rangle \langle d^n, K^\ell \rangle 
         ~=~ \langle d^n, E^i F^j  \rangle   \\
         &= \langle c^n, E^i \rangle \langle b^n, F^j \rangle + \langle d^n, E^i \rangle \langle d^n, F^j \rangle 
         ~=~ \delta_{i,0} \delta_{j, 0}.
    \end{align*}
    We have thus shown that $ \phi $ vanishes on $ b^n $ and $ c^n $. 
    Now, since $ \epsilon $ is an algebra map, we have that $\langle 1, E^i F^j K^\ell \rangle = \epsilon ( E^i F^j K^\ell ) =    \epsilon ( E )^i \epsilon ( F ) ^j \epsilon ( K ) ^ \ell = \delta_{i,0} \delta_{j,0} $. 
    Thus, $\phi$ also vanishes on $ a^n - 1 $ and $ d^n - 1 $.
   \end{proof}
   
At this point, we want to establish that the duality just formed between $\overline{\mc O_q(SL_2)}$ and $u_q(\mf{sl}_2)$ is a perfect duality.
We do this by showing that $\overline \phi$ is surjective, for which we will need the following technical computation.

For the basis $\{E^i F^j K^\ell\}$ of $u_q(\mf{sl}_2)$, we let $\{ p_{i,j,\ell} \}$ denote the dual basis of $u_q(\mf{sl}_2)^*$.
Because $K^n = 1$ in $u_q(\mf{sl}_2)$, we will take the last argument of these basis elements modulo $n$.   

Now via elementary computations we have in terms of the dual basis $\{ p_{i,j,\ell} \}$ of $u_q(\mf{sl}_2)^*$, that
   \begin{equation} \label{eq:oqdualbasis}
      B^s C^t D^r = [s]_q! [t]_q!\sum_{\ell=0}^{n-1} q^{-\ell(r+s-t) - rs} p_{s, t, \ell}.
   \end{equation}
  
\begin{proposition} \label{prop:uqsl2perfectdual}
  The map $\overline \phi: \overline{\mc O_q(SL_2)} \lra u_q(\mf{sl}_2)^*$ is surjective, and hence is an isomorphism.
  Thus, the bilinear form $\langle u,x \rangle = \overline \phi (u)(x)$ realizes a perfect duality between $\overline{\mc O_q(SL_2)}$ and $u_q(\mf{sl}_2)$.
\end{proposition}

  \begin{proof}
    We show that each basis element $p_{i,j,k}$ of $u_q(\mf{sl}_2)$ is in the image of $\overline \phi$.
    In particular, for fixed integers $ 0 \leq s, t, k \leq n-1 $, we show that 
   \[
      n \ [s]_q! \ [t]_q! \ p_{s, t, k} = \sum _ { r = 0 }^{n-1} q^{ ( k+s ) r + ( s-t ) k  } B^s C^t D^r .
   \]      
   We compute via \eqref{eq:oqdualbasis}
   \begin{align*}
     \sum _{r=0} ^ {n-1} q^{(k+s)r + (s-t)k} B^s C^t D^r 
       &= \sum _{r=0} ^ {n-1} q^{(k+s)r + (s-t)k} [s]_q! [t]_q! \sum _ {\ell=0} ^{n-1} q^ { -\ell ( r + s - t ) - rs } p_{s, t, \ell}   \\
     &= [s]_q! [t]_q! \sum _{\ell=0} ^ {n-1} \left( \sum _ {r=0} ^ {n-1} q^{ (k-\ell) ( r + s - t ) } \right) p_{s, t, \ell} . 
    \end{align*}
    If $ k \neq \ell $, then since $ q ^ {k-\ell} $ is an $n^{th}$ root of unity not equal to 1, $ \sum_{r=0}^{n-1} q^{(k-\ell)(r+s-t)} = 0$. 
    On the other hand, if $ k = \ell $, then $ \sum_{r=0}^{n-1} q^{(k-\ell)(r+s-t)} = n $.
  \end{proof}
  
Now that we have established the fact that $u_q(\mf{sl}_2)^* \cong \overline{ \mc O_q(SL_2)}$, we can prove the following.

\begin{theorem} \label{thm:uqdouble}
  The Drinfel'd double $D(u_q(\mf{sl}_2))$ of $u_q(\mf{sl}_2)$ is generated as an algebra by $a, \ b, \ c, \ d, \ E, \ F, \ K$ subject to the relations
  \begin{gather*}
    a^n = d^n = K^n = 1, \quad
    b^n = c^n = E^n = F^n = 0, \\
    ba = qab, 
    \quad db = qbd, 
    \quad ca = qac, 
    \quad dc = qcd, 
    \quad bc = cb, 
    \quad ad = q^{-1} bc + 1, 
    \\
    KE = q^2 EK, \quad KF = q^{-2} FK, \quad EF - FE = \frac { K- K^{-1} } { q - q^{-1} } , \\
    Ka = aK, \quad 
       Kb = q^{-2}bK, \quad 
       Kc = q^2 cK, \quad
       Kd = dK, \\
    Ea = q^{-1}aE - q^{-1}c, \quad
       Eb = q^{-1}bE + q^{-1}aK - q^{-1} d, \quad
       Ec = qcE, \quad
       Ed = qdE + qcK, \\
    Fa = q^{-1}aF + b, \quad
       Fb = qbF, \quad
       Fc = q^{-1}cF = aK^{-1} + d, \quad
       Fd = qdF - q^2 b K^{-1}
  \end{gather*}
  The comultiplication and counit are given by 
  \begin{gather*}
     \Delta ( a ) = a \tens a + c \tens b , \quad
     \Delta ( b ) = b \tens a + d \tens b , \quad
     \Delta ( c ) = a \tens c + c \tens d , \quad
     \Delta ( d ) = b \tens c + d \tens d ,  \\
     \Delta ( K ) = K \tens K , \quad
     \Delta ( E ) = K \tens E + E \tens 1 , \quad
     \Delta ( F ) = 1 \tens F + F \tens K^{-1} ,    \\
     \epsilon ( a ) = \epsilon ( d ) = \epsilon ( K ) = 1 , \quad
     \epsilon ( b ) = \epsilon ( c ) = \epsilon ( E ) = \epsilon ( F ) = 0 .
  \end{gather*}
  The antipode is given by 
  \begin{gather*}
     S ( a ) = d, \quad 
     S(b) = -q^{-1} b, \quad 
     S(c) = -q c, \quad 
     S(d) = a , \quad
     S ( K ) = K^{-1}, \quad
     S(E) = -EK^{-1}, \quad
     S(F) = -KF .
  \end{gather*}
\end{theorem}

As pointed out above, the generator $ a $ (or $d$) could be eliminated from the presentation, using the relation $ ad = q^{-1} b c + 1 $ and the fact that $a$ and $d$ are invertible. 
While doing so would significantly lower the number of relations, it would complicate both the relations between generators of $u_q(\mf{sl}_2)$ and $\overline{\mc O_q(SL_2)}$ and the comultiplication of the latter. 

  \begin{proof}[Proof of Theorem~\ref{thm:uqdouble}]
    The comultiplication and antipode and most of the relations of the generators follow from \eqref{eq:doublecomult} and Lemma~\ref{lem:gens}.
    For the relations involving elements of both $u_q(\mf{sl}_2)$ and its dual, we use \eqref{eq:doublemixed} and the perfect duality established in Proposition~\ref{prop:uqsl2perfectdual}.
    Note that 
    \begin{gather*}
       \Delta^2 ( a ) = a \tens a \tens a + a \tens b \tens c + b \tens c \tens a + b \tens d \tens c , \quad
       \Delta^2 ( b ) = a \tens a \tens b + a \tens b \tens d + b \tens c \tens b + b \tens d \tens d  ,  \\
       \Delta^2 ( c ) =   c \tens a \tens a + c \tens b \tens c + d \tens c \tens a + d \tens d \tens c  , \quad
       \Delta^2 ( d ) = c \tens a \tens b + c \tens b \tens d + d \tens c \tens b + d \tens d \tens d , \\
       \Delta^2 ( E ) = 1 \tens 1 \tens E + 1 \tens E \tens K + E \tens K \tens K  , \quad
       \Delta^2 ( K ) = K \tens K \tens K ,  \\
       \Delta^2 ( F ) = K^{-1} \tens K^{-1} \tens F + K^{-1} \tens F \tens 1 + F \tens 1 \tens 1 ,  \quad
       S^{-1} ( E ) = - K^{-1} E , \quad
       S^{-1} ( F ) = - F K .
    \end{gather*}
    For example, we have
    \begin{align*}
       Ea &= \langle \com a 1 , S^{-1} ( \com E 3 ) \rangle \langle \com a 3, \com E 1 \rangle \com a 2 \com E 2   \\
       &= \hspace{3mm} %
          \langle a, -K^{-1}E \rangle \langle a, 1 \rangle a 1 + %
          \langle a, K^{-1} \rangle \langle a, 1 \rangle a E + %
          \langle a, K^{-1} \rangle \langle a, E \rangle a K  \\
       & \hspace{4mm} +  %
          \langle a, -K^{-1}E \rangle \langle c, 1 \rangle b 1 + %
          \langle a, K^{-1} \rangle \langle c, 1 \rangle b E + %
          \langle a, K^{-1} \rangle \langle c, E \rangle b K  \\
       & \hspace{4mm} +  %
          \langle b, -K^{-1}E \rangle \langle a, 1 \rangle c 1 + %
          \langle b, K^{-1} \rangle \langle a, 1 \rangle c E + %
          \langle b, K^{-1} \rangle \langle a, E \rangle c K  \\
       & \hspace{4mm} +  %
          \langle b, -K^{-1}E \rangle \langle c, 1 \rangle d 1 + %
          \langle b, K^{-1} \rangle \langle c, 1 \rangle d E + %
          \langle b, K^{-1} \rangle \langle c, E \rangle d K     \\
       &= q^{-1} aE - q^{-1} c  .
    \end{align*}
    \begin{align*}
       Fa 
       &= \hspace{3mm} %
          \langle a, -FK \rangle \langle a, K^{-1} \rangle a K^{-1} + %
          \langle a, 1 \rangle \langle a, K^{-1} \rangle a F + %
          \langle a, 1 \rangle \langle a, F \rangle a 1  \\
       & \hspace{4mm} +  %
          \langle a, -FK \rangle \langle c, K^{-1} \rangle b K^{-1} + %
          \langle a, 1 \rangle \langle c, K^{-1} \rangle b F + %
          \langle a, 1 \rangle \langle c, F \rangle b 1  \\
       & \hspace{4mm} +  %
          \langle b, -FK \rangle \langle a, K^{-1} \rangle c K^{-1} + %
          \langle b, 1 \rangle \langle a, K^{-1} \rangle c F + %
          \langle b, 1 \rangle \langle a, F \rangle c 1  \\
       & \hspace{4mm} +  %
          \langle b, -FK \rangle \langle c, K^{-1} \rangle d K^{-1} + %
          \langle b, 1 \rangle \langle c, K^{-1} \rangle d F + %
          \langle b, 1 \rangle \langle c, F \rangle d 1     \\
       &= q^{-1} aF + b  .
    \end{align*}
    The rest of the relations follow similarly.
        
  \end{proof}

\bibliography{articlebib}

\end{document}